\documentclass[12pt]{article}

\usepackage[utf8]{inputenc}
\usepackage[T1]{fontenc}
\usepackage{amsfonts,amssymb,amsmath,amsthm}
\usepackage[margin=2.8cm]{geometry}
\usepackage{graphicx}
\usepackage{calc}

\usepackage{tikz}
\usetikzlibrary{patterns}
\usetikzlibrary{decorations}
\usetikzlibrary{decorations.pathmorphing}
\usetikzlibrary{decorations.pathreplacing}
\usetikzlibrary{decorations.shapes}
\usetikzlibrary{decorations.text}
\usetikzlibrary{decorations.markings}
\usetikzlibrary{decorations.fractals}
\usetikzlibrary{decorations.footprints}

\usepackage{dirtytalk}

\usepackage{imakeidx}
\usepackage{caption}
\usepackage{subcaption}
\usepackage{nomencl}
\usepackage{float}
\usepackage{stmaryrd}
\usepackage{subfiles}
\usepackage{ifthen}
\usepackage{lmodern}
\usetikzlibrary{intersections}
\usepackage{hyperref}
\hypersetup{
    colorlinks=true,
    linkcolor=teal,
    breaklinks=true, 
urlcolor= teal, 
citecolor= teal
    }

\makeindex 

\newcommand{\A}{\mathcal{A}}
\newcommand{\B}{\mathcal{B}}
\newcommand{\I}{\mathcal{I}}
\newcommand{\R}{\mathbb{R}}
\newcommand{\T}{\mathcal{T}}
\newcommand{\W}{\mathcal{W}}
\newcommand{\Z}{\mathbb{Z}}
\newcommand{\Q}{\mathbb{Q}}
\newcommand{\N}{\mathbb{N}}
\newcommand{\indi}{\mathbf{1}_{\left[0,1\right]}}
\newcommand{\pert}{\{1,\dots,6\}}
\newcommand{\genre}{\{1,\dots,g\}}

\newcommand{\id}{\operatorname{Id}}
\newcommand{\bijlab}{\mathfrak{S}_6}

\newcommand{\zkkt}{z_{\mbox{\scriptsize KKT}}}
\newcommand{\Zkkt}{\mathcal{Z}_{\mbox{\scriptsize KKT}}}
\newcommand{\zlmo}{z_{\mbox{\scriptsize LMO}}}

\theoremstyle{plain}
\newtheorem{thm}{Theorem}[section]
\newtheorem{lem}[thm]{Lemma}

\newtheorem{prop}[thm]{Proposition}

\theoremstyle{definition}
\newtheorem{defi}[thm]{Definition}
\newtheorem{exemple}[thm]{Example}
\newtheorem{nota}[thm]{Notation}

\theoremstyle{remark}
\newtheorem{rque}[thm]{Remark}

\newcommand{\tetra}{\begin{tikzpicture}\useasboundingbox (-.1,0) rectangle (.7,.5);
\draw (0,0) -- (.6,0) -- (.3,.5) -- (0,0);
\draw (.6,0) -- (.3,.2) -- (.3,.5) (0,0) -- (.3,.2);
\fill (0,0) circle (1.5pt) (.6,0) circle (1.5pt) (.3,.5) circle (1.5pt) (.3,.2) circle (1.5pt);
\end{tikzpicture}}
\newcommand{\tetraeq}{\begin{tikzpicture}\useasboundingbox (-.1,0) rectangle (.7,0.5);
		\begin{scope}[yshift=-0.1cm]
		\draw (0,0) -- (.6,0) -- (.3,.5) -- (0,0);
		\draw (.6,0) -- (.3,.2) -- (.3,.5) (0,0) -- (.3,.2);
		\fill (0,0) circle (1.5pt) (.6,0) circle (1.5pt) (.3,.5) circle (1.5pt) (.3,.2) circle (1.5pt);
		\end{scope}
\end{tikzpicture}}
\newcommand{\tatata}{\begin{tikzpicture}\useasboundingbox (-.1,0) rectangle (.7,.5);
\draw (.1,0) -- (.5,0)  (.1,.5) -- (.5,.5);
\draw (.1,0) ..  controls (-.05,.25) .. (.1,.5);
\draw (.1,0) ..  controls (.25,.25) .. (.1,.5);
\draw (.5,0) ..  controls (.35,.25) .. (.5,.5);
\draw (.5,0) ..  controls (.65,.25) .. (.5,.5);
\fill (.1,0) circle (1.5pt) (.1,.5) circle (1.5pt) (.5,0) circle (1.5pt) (.5,.5) circle (1.5pt);
\end{tikzpicture}}

\newcommand{\cycletetra}{\begin{tikzpicture}[scale=0.4]
\useasboundingbox (-1.2,0.1) rectangle (1.2,1);
\draw (-1,0)--(1,0) node[midway,sloped]{\scriptsize$>$};
\draw (1,0)--(0,1) node[midway,sloped]{\scriptsize$<$};
\draw (0,1)--(-1,0) node[midway,sloped]{\scriptsize$<$};
\end{tikzpicture}}

\newcommand{\cycletheta}{\begin{tikzpicture}[scale=0.4]
\useasboundingbox (-1.2,-0.25) rectangle (1.2,0.25);
\draw (-1,0).. controls (0,0.5).. (1,0) node[midway,sloped]{\scriptsize$>$};
\draw (-1,0)..controls(0,-0.5)..(1,0) node[midway,sloped]{\scriptsize$<$};
\end{tikzpicture}}

\newcommand{\antisym}{\begin{tikzpicture}[scale=0.6]
\useasboundingbox (-1,-1) rectangle (4,1);
\coordinate (a) at (0,0);
\coordinate (b) at (30:1);
\coordinate (c) at (150:1);
\coordinate (d) at (0,-1);
\draw[dashed] (a) circle (1);
\draw [very thick] (d) -- (a) -- (c);
\draw [very thick] (a) --(b);
\draw [very thick] (a) node{$\bullet$};

\draw  [very thick] (1.5,-0.2)--(1.5,0.2);
\draw [very thick] (1.3,0)--(1.7,0);

\begin{scope}[xshift=3cm]
\coordinate (e) at (0,0);

\coordinate (f) at (30:1);
\coordinate (g) at (150:1);
\coordinate (h) at (0,-1);
\draw[dashed] (e) circle (1);
\draw [very thick] (e).. controls (1,0).. (g);
\draw [very thick] (e).. controls (-1,0).. (f);
\draw [very thick] (e) --(h);
\draw [very thick] (e) node{$\bullet$};

\draw [very thick] (e) node{$\bullet$};
\end{scope}
\end{tikzpicture}}

\newcommand{\Jacobi}{\begin{tikzpicture}[scale=0.6]
\useasboundingbox (-2,-1) rectangle (9,1);

\draw[dashed] (0,0) circle (1);

\coordinate (a) at (0,0);
\coordinate (b) at (0,-0.5);
\coordinate (c) at (60:1);
\coordinate (d) at (120:1);
\coordinate (e) at (0,-1);
\coordinate (f) at (220:1);

\draw [very thick] (e)-- (a)--(c) ;
\draw [very thick] (a) -- (d);
\draw [very thick] (f) -- (b);
\draw (a) node{$\bullet$};
\draw (b) node{$\bullet$};

\draw  [very thick] (1.5,-0.2)--(1.5,0.2);
\draw [very thick] (1.3,0)--(1.7,0);

\begin{scope}[xshift=3cm]
\draw[dashed] (0,0) circle (1);

\coordinate (a) at (0,0);
\coordinate (b) at (60:0.5);
\coordinate (c) at (60:1);
\coordinate (d) at (120:1);
\coordinate (e) at (0,-1);
\coordinate (f) at (220:1);

\draw [very thick] (e) -- (a)--(c) ;
\draw [very thick] (a) -- (d);
\draw [very thick] (f) .. controls (1,0.25) .. (b);

\draw (a) node{$\bullet$};
\draw (b) node{$\bullet$};

\draw  [very thick] (1.5,-0.2)--(1.5,0.2);
\draw [very thick] (1.3,0)--(1.7,0);
\end{scope}
\begin{scope}[xshift=6cm]
\draw[dashed] (0,0) circle (1);

\coordinate (a) at (0,0);
\coordinate (b) at (120:0.5);
\coordinate (c) at (60:1);
\coordinate (d) at (120:1);
\coordinate (e) at (0,-1);
\coordinate (f) at (220:1);

\draw [very thick] (e) -- (a)--(c) ;
\draw [very thick] (a) -- (d);
\draw [very thick] (f) .. controls (1.4,0.25) .. (b);

\draw (a) node{$\bullet$};
\draw (b) node{$\bullet$};

\end{scope}

\end{tikzpicture}}

\newcommand{\fourgraphs}{\begin{tikzpicture}
\useasboundingbox (-2,-1) rectangle (11,2);
\draw[very thick] (-0.88,-0.5) -- (0.88,-0.5) node[midway,sloped]{$>$} ;
\draw (0,-1) node{$T_1$};
\draw[very thick] (0.88,-0.5) -- (0,0) node[midway,sloped]{$<$};
\draw[very thick] (0,0) -- (-0.88,-0.5) node[midway,sloped]{$<$} ;
\draw[very thick] (0,1) -- (-0.88,-0.5) node[midway,sloped]{$<$} ;
\draw[very thick] (0,1) -- (0,0) node[midway,sloped]{$>$};
\draw[very thick] (0,1) --(0.88,-0.5) node[midway,sloped]{$>$} ;
\draw[very thick] (0,1) node[above]{\scriptsize$1$};
\draw[very thick] (-0.88,-0.5)++(-0.15,-0.15) node{\scriptsize$2$};
\draw[very thick] (0.88,-0.5)++(0.15,-0.15) node{\scriptsize$3$};
\draw[very thick] (0,0)++(0.15,0.15) node{\scriptsize$4$};

\draw[very thick] (2.12,-0.5) -- (3.88,-0.5) node[midway,sloped]{$<$};
\draw (3,-1) node{$T_2$};
\draw[very thick] (3.88,-0.5) -- (3,0) node[midway,sloped]{$>$};
\draw[very thick] (3,0) -- (2.12,-0.5) node[midway,sloped]{$>$} ;
\draw[very thick] (3,1) -- (2.12,-0.5) node[midway,sloped]{$>$} ;
\draw[very thick] (3,1) -- (3,0) node[midway,sloped]{$<$};
\draw[very thick] (3,1) --(3.88,-0.5) node[midway,sloped]{$<$} ;

\begin{scope}[xshift=3cm]
\draw[very thick] (0,1) node[above]{\scriptsize$1$};
\draw[very thick] (-0.88,-0.5)++(-0.15,-0.15) node{\scriptsize$2$};
\draw[very thick] (0.88,-0.5)++(0.15,-0.15) node{\scriptsize$3$};
\draw[very thick] (0,0)++(0.15,0.15) node{\scriptsize$4$};
\end{scope}

\draw[very thick] (5,-0.5) -- (6.76,-0.5) node[midway,sloped]{$>$};
\draw (5.88,-1) node{$W_1$};
\draw[very thick] (5,1) -- (6.76,1) node[midway,sloped]{$>$};
\draw[very thick] (5,-0.5) .. controls (5.5,0.2) ..(5,1) node[midway,sloped]{$<$};
\draw[very thick] (5,-0.5) .. controls (4.5,0.2) ..(5,1) node[midway,sloped]{$<$};
\draw[very thick] (6.76,-0.5) .. controls (7.26,0.2) ..(6.76,1) node[midway,sloped]{$>$};
\draw[very thick] (6.76,-0.5) .. controls (6.26,0.2) ..(6.76,1) node[midway,sloped]{$<$};

\draw[very thick] (5,-0.5)++(-0.15,-0.15) node{\scriptsize$2$};
\draw[very thick] (5,1)++(-0.15,0.15) node{\scriptsize$1$};
\draw[very thick] (6.76,-0.5)++(0.15,-0.15) node{\scriptsize$4$};
\draw[very thick] (6.76,1)++(0.15,0.15) node{\scriptsize$3$};

\draw[very thick] (8,-0.5) -- (9.76,-0.5) node[midway,sloped]{$<$};
\draw (8.88,-1) node{$W_2$};
\draw[very thick] (8,1) -- (9.76,1) node[midway,sloped]{$<$};
\draw[very thick] (8,-0.5) .. controls (8.5,0.2) ..(8,1) node[midway,sloped]{$>$};
\draw[very thick] (8,-0.5) .. controls (7.5,0.2) ..(8,1) node[midway,sloped]{$>$};
\draw[very thick] (9.76,-0.5) .. controls (10.26,0.2) ..(9.76,1) node[midway,sloped]{$<$};
\draw[very thick] (9.76,-0.5) .. controls (9.26,0.2) ..(9.76,1) node[midway,sloped]{$>$};

\begin{scope}[xshift=3cm]
\draw[very thick] (5,-0.5)++(-0.15,-0.15) node{\scriptsize$2$};
\draw[very thick] (5,1)++(-0.15,0.15) node{\scriptsize$1$};
\draw[very thick] (6.76,-0.5)++(0.15,-0.15) node{\scriptsize$4$};
\draw[very thick] (6.76,1)++(0.15,0.15) node{\scriptsize$3$};
\end{scope}

\end{tikzpicture}}

\newcommand{\Tunnew}{\begin{tikzpicture}[scale=1.5]
\useasboundingbox (-1,-0.7) rectangle (1,1.2);
\draw[very thick] (-0.88,-0.5) -- (0.88,-0.5) node[midway,sloped]{$>$} node[near start, below,scale=0.7]{\scriptsize$e_1$} node[near end, below,scale=0.7]{\scriptsize$e_2$};
\draw[very thick] (0.88,-0.5) -- (0,0) node[midway,sloped]{$<$} node[near end, below,scale=0.7]{\scriptsize$d_2$};
\draw (0.55,-0.4) node[scale=0.7]{\scriptsize$d_1$};
\draw[very thick] (0,0) -- (-0.88,-0.5) node[midway,sloped]{$<$} node[near start, below,scale=0.7]{\scriptsize$a_1$} ;
\draw (-0.55,-0.4) node[scale=0.7]{\scriptsize$a_2$};
\draw[very thick] (0,1) -- (-0.88,-0.5) node[midway,sloped]{$<$} ;
\draw (-0.3,0.75) node[scale=0.7]{\scriptsize$b_1$};
\draw (-0.7,0) node[scale=0.7]{\scriptsize$b_2$};
\draw[very thick] (0,1) -- (0,0) node[midway,sloped]{$>$};
\draw (0.08,0.75) node[scale=0.7]{\scriptsize$f_1$};
\draw (0.08,0.25) node[scale=0.7]{\scriptsize$f_2$};
\draw[very thick] (0,1) --(0.88,-0.5) node[midway,sloped]{$>$}  ;
\draw (0.3,0.75) node[scale=0.7]{\scriptsize$c_1$};
\draw (0.7,0) node[scale=0.7]{\scriptsize$c_2$};
\draw (-0.88,-0.7) node[scale=0.7]{\scriptsize$v_2$};
\draw (0.88,-0.7) node[scale=0.7]{\scriptsize$v_3$};
\draw (0,1.2) node[scale=0.7]{\scriptsize$v_1$};
\draw (-0.1,0.1) node[scale=0.7]{\scriptsize$v_4$};
\draw (-0.88,-0.5) node{$\bullet$};
\draw (0.88,-0.5) node{$\bullet$};
\draw (0,0) node{$\bullet$};
\draw (0,1) node{$\bullet$};
\end{tikzpicture}}

\newcommand{\typeoneandtwo}{\begin{tikzpicture}
\useasboundingbox (-1,-1.1) rectangle (1,1.1);
\draw[very thick] (-0.88,-0.5) -- (0.88,-0.5) node[midway,sloped]{$>$};
\draw[very thick] (0.88,-0.5) -- (0,0) node[midway,sloped]{$<$};
\draw[very thick] (0,0) -- (-0.88,-0.5) node[midway,sloped]{$<$} ;
\draw[very thick] (0,1) -- (-0.88,-0.5) node[midway,sloped]{$<$} ;
\draw[very thick] (0,1) -- (0,0) node[midway,sloped]{$>$};
\draw[very thick] (0,1) --(0.88,-0.5) node[midway,sloped]{$>$} ;
\draw [dashed] (0,0.75) circle (0.4);
\draw [dashed] (0,-0.25) circle (0.5);
\draw (1,0.75) node{\scriptsize Type $2$};
\draw (1,-0.75) node{\scriptsize Type $1$};
\end{tikzpicture}}

\newcommand{\doubleedge}{\begin{tikzpicture}[scale=1]
\useasboundingbox (4,-0.7) rectangle (7.76,1.2);

\draw[very thick] (5,-0.5) -- (6.76,-0.5) node[midway,sloped]{$>$} node[midway, below]{\scriptsize$e_{24}$};
\draw[very thick] (5,1) -- (6.76,1) node[midway,sloped]{$>$} node[midway,above]{\scriptsize$e_{13}$};
\draw[very thick] (5,-0.5) .. controls (5.5,0.2) ..(5,1) node[midway,sloped]{$<$};
\draw[very thick] (5,-0.5) .. controls (4.5,0.2) ..(5,1) node[midway,sloped]{$<$};
\draw[very thick] (6.76,-0.5) .. controls (7.26,0.2) ..(6.76,1) node[midway,sloped]{$>$};
\draw (7.5,0.25) node{\scriptsize$e_{43}$};
\draw[very thick] (6.76,-0.5) .. controls (6.26,0.2) ..(6.76,1) node[midway,sloped]{$<$};
\draw (6.05,0.25) node{\scriptsize$e_{34}$};
\draw (5,-0.7) node {\scriptsize$2$};
\draw (5,1.3) node{\scriptsize$1$};
\draw (6.76,-0.7) node {\scriptsize$4$};
\draw (6.76,1.3) node {\scriptsize$3$};
\draw (5,-0.5) node{$\bullet$};
\draw (6.76,-0.5) node{$\bullet$};
\draw (5,1) node{$\bullet$};
\draw (6.76,1) node{$\bullet$};
\end{tikzpicture}}

\newcommand{\Wunnew}{\begin{tikzpicture}[scale=1.5]

\useasboundingbox (4.75,-0.7) rectangle (7,1.2);

\draw[very thick] (5,-0.5) -- (6.76,-0.5) node[midway,sloped]{$>$} node[near start, below,scale=0.7]{\scriptsize $e_1$} node[near end, below,scale=0.7]{\scriptsize$e_2$};
\draw[very thick] (5,1) -- (6.76,1) node[midway,sloped]{$>$} node[near start, above,scale=0.7]{\scriptsize$f_1$} node[near end,above,scale=0.7]{\scriptsize$f_2$};
\draw[very thick] (5,-0.5) .. controls (5.5,0.2) ..(5,1) node[midway,sloped]{$<$};
\draw (5.5,0) node[scale=0.7]{\scriptsize$c_2$};
\draw (5.5,0.5) node[scale=0.7]{\scriptsize$c_1$};
\draw[very thick] (5,-0.5) .. controls (4.5,0.2) ..(5,1) node[midway,sloped]{$<$};
\draw (4.5,0) node[scale=0.7]{\scriptsize$b_2$};
\draw (4.5,0.5) node[scale=0.7]{\scriptsize$b_1$};
\draw[very thick] (6.76,-0.5) .. controls (7.26,0.2) ..(6.76,1) node[midway,sloped]{$>$};
\draw (7.26,0) node[scale=0.7]{\scriptsize$d_1$};
\draw (7.26,0.5) node[scale=0.7]{\scriptsize$d_2$};
\draw[very thick] (6.76,-0.5) .. controls (6.26,0.2) ..(6.76,1) node[midway,sloped]{$<$};
\draw (6.26,0) node[scale=0.7]{\scriptsize$a_2$};
\draw (6.26,0.5) node[scale=0.7]{\scriptsize$a_1$};
\draw (5,-0.7) node[scale=0.7]{\scriptsize$v_2$};
\draw (5,1.2) node[scale=0.7]{\scriptsize$v_1$};
\draw (6.76,-0.7) node[scale=0.7]{\scriptsize$v_4$};
\draw (6.76,1.2) node[scale=0.7]{\scriptsize$v_3$};
\draw (5,-0.5) node{$\bullet$};
\draw (6.76,-0.5) node{$\bullet$};
\draw (5,1) node{$\bullet$};
\draw (6.76,1) node{$\bullet$};
\end{tikzpicture}}

\newcommand{\Wunbisnew}
{\begin{tikzpicture}[scale=1.5]
\useasboundingbox (4.75,-0.7) rectangle (7,1.2);

\draw[very thick] (5,-0.5) -- (6.76,-0.5) node[midway,sloped]{$>$} node[near start, below,scale=0.7]{\scriptsize$f_1$} node[near end, below,scale=0.7]{\scriptsize$f_2$};
\draw[very thick] (5,1) -- (6.76,1) node[midway,sloped]{$>$} node[near start, above,scale=0.7]{\scriptsize$e_1$} node[near end,above,scale=0.7]{\scriptsize$e_2$};
\draw[very thick] (5,-0.5) .. controls (5.5,0.2) ..(5,1) node[midway,sloped]{$<$};
\draw (5.5,0) node[scale=0.7]{\scriptsize$b_2$};
\draw (5.5,0.5) node[scale=0.7]{\scriptsize$b_1$};
\draw[very thick] (5,-0.5) .. controls (4.5,0.2) ..(5,1) node[midway,sloped]{$<$};
\draw (4.5,0) node[scale=0.7]{\scriptsize$c_2$};
\draw (4.5,0.5) node[scale=0.7]{\scriptsize$c_1$};
\draw[very thick] (6.76,-0.5) .. controls (7.26,0.2) ..(6.76,1) node[midway,sloped]{$>$};
\draw (7.26,0) node[scale=0.7]{\scriptsize$d_1$};
\draw (7.26,0.5) node[scale=0.7]{\scriptsize$d_2$};
\draw[very thick] (6.76,-0.5) .. controls (6.26,0.2) ..(6.76,1) node[midway,sloped]{$<$};
\draw (6.26,0) node[scale=0.7]{\scriptsize$a_2$};
\draw (6.26,0.5) node[scale=0.7]{\scriptsize$a_1$};
\draw (5,-0.7) node[scale=0.7]{\scriptsize$v_2$};
\draw (5,1.2) node[scale=0.7]{\scriptsize$v_1$};
\draw (6.76,-0.7) node[scale=0.7]{\scriptsize$v_4$};
\draw (6.76,1.2) node[scale=0.7]{\scriptsize$v_3$};
\draw (5,-0.5) node{$\bullet$};
\draw (6.76,-0.5) node{$\bullet$};
\draw (5,1) node{$\bullet$};
\draw (6.76,1) node{$\bullet$};
\end{tikzpicture}}

\newcommand{\TunIHXnew}
{\begin{tikzpicture}[scale=1.5]
\useasboundingbox (-1,-0.7) rectangle (1,1.2);
\draw[very thick] (-0.88,-0.5) -- (0.88,-0.5) node[midway,sloped]{$>$} node[near start, below,scale=0.7]{\scriptsize$e_1$} node[near end, below,scale=0.7]{\scriptsize$e_2$};
\draw[very thick] (0.88,-0.5) -- (0,0) node[midway,sloped]{$<$} node[near end, below,scale=0.7]{\scriptsize$d_2$};
\draw (0.55,-0.4) node[scale=0.7]{\scriptsize$d_1$};
\draw[very thick] (0,0) -- (-0.88,-0.5) node[midway,sloped]{$<$} node[near start, below,scale=0.7]{\scriptsize$a_1$} ;
\draw (-0.55,-0.4) node[scale=0.7]{\scriptsize$a_2$};
\draw[very thick] (0,1) -- (-0.88,-0.5) node[midway,sloped]{$<$} ;
\draw (-0.3,0.75) node[scale=0.7]{\scriptsize$c_1$};
\draw (-0.7,0) node[scale=0.7]{\scriptsize$c_2$};
\draw[very thick] (0,1) -- (0,0) node[midway,sloped]{$>$};
\draw (0.08,0.75) node[scale=0.7]{\scriptsize$f_1$};
\draw (0.08,0.25) node[scale=0.7]{\scriptsize$f_2$};
\draw[very thick] (0,1) --(0.88,-0.5) node[midway,sloped]{$>$}  ;
\draw (0.3,0.75) node[scale=0.7]{\scriptsize$b_1$};
\draw (0.7,0) node[scale=0.7]{\scriptsize$b_2$};
\draw (-0.88,-0.7) node[scale=0.7]{\scriptsize$v_2$};
\draw (0.88,-0.7) node[scale=0.7]{\scriptsize$v_3$};
\draw (0,1.2) node[scale=0.7]{\scriptsize$v_1$};
\draw (-0.1,0.1) node[scale=0.7]{\scriptsize$v_4$};
\draw (-0.88,-0.5) node{$\bullet$};
\draw (0.88,-0.5) node{$\bullet$};
\draw (0,0) node{$\bullet$};
\draw (0,1) node{$\bullet$};
\end{tikzpicture}}

\newcommand{\Tunbisnew}{\begin{tikzpicture}[scale=1.5]
\useasboundingbox (-1,-0.7) rectangle (1,1.2);
\draw[very thick] (-0.88,-0.5) -- (0.88,-0.5) node[midway,sloped]{$>$} node[near start, below,scale=0.7]{\scriptsize$c_1$} node[near end, below,scale=0.7]{\scriptsize$c_2$};
\draw[very thick] (0.88,-0.5) -- (0,0) node[midway,sloped]{$<$} node[near end, below,scale=0.7]{\scriptsize$f_2$};
\draw (0.55,-0.4) node[scale=0.7]{\scriptsize$f_1$};
\draw[very thick] (0,0) -- (-0.88,-0.5) node[midway,sloped]{$<$} node[near start, below,scale=0.7]{\scriptsize$d_1$} ;
\draw (-0.55,-0.4) node[scale=0.7]{\scriptsize$d_2$};
\draw[very thick] (0,1) -- (-0.88,-0.5) node[midway,sloped]{$<$} ;
\draw (-0.3,0.75) node[scale=0.7]{\scriptsize$e_1$};
\draw (-0.7,0) node[scale=0.7]{\scriptsize$e_2$};
\draw[very thick] (0,1) -- (0,0) node[midway,sloped]{$>$};
\draw (0.08,0.75) node[scale=0.7]{\scriptsize$a_1$};
\draw (0.08,0.25) node[scale=0.7]{\scriptsize$a_2$};
\draw[very thick] (0,1) --(0.88,-0.5) node[midway,sloped]{$>$}  ;
\draw (0.3,0.75) node[scale=0.7]{\scriptsize$b_1$};
\draw (0.7,0) node[scale=0.7]{\scriptsize$b_2$};
\draw (-0.88,-0.7) node[scale=0.7]{\scriptsize$v_2$};
\draw (0.88,-0.7) node[scale=0.7]{\scriptsize$v_3$};
\draw (0,1.2) node[scale=0.7]{\scriptsize$v_1$};
\draw (-0.1,0.1) node[scale=0.7]{\scriptsize$v_4$};
\draw (-0.88,-0.5) node{$\bullet$};
\draw (0.88,-0.5) node{$\bullet$};
\draw (0,0) node{$\bullet$};
\draw (0,1) node{$\bullet$};
\end{tikzpicture}}

\newcommand{\TunIHXbisnew}{\begin{tikzpicture}[scale=1.5]
\useasboundingbox (-1,-0.7) rectangle (1,1.2);
\draw[very thick] (-0.88,-0.5) -- (0.88,-0.5) node[midway,sloped]{$>$} node[near start, below, scale=0.7]{\scriptsize$b_1$} node[near end, below,scale=0.7]{\scriptsize$b_2$};
\draw[very thick] (0.88,-0.5) -- (0,0) node[midway,sloped]{$<$} node[near end, below,scale=0.7]{\scriptsize$f_2$};
\draw (0.55,-0.4) node[scale=0.7]{\scriptsize$f_1$};
\draw[very thick] (0,0) -- (-0.88,-0.5) node[midway,sloped]{$<$} node[near start, below,scale=0.7]{\scriptsize$d_1$} ;
\draw (-0.55,-0.4) node[scale=0.7]{\scriptsize$d_2$};
\draw[very thick] (0,1) -- (-0.88,-0.5) node[midway,sloped]{$<$} ;
\draw (-0.3,0.75) node[scale=0.7]{\scriptsize$e_1$};
\draw (-0.7,0) node[scale=0.7]{\scriptsize$e_2$};
\draw[very thick] (0,1) -- (0,0) node[midway,sloped]{$>$};
\draw (0.08,0.75) node[scale=0.7]{\scriptsize$a_1$};
\draw (0.08,0.25) node[scale=0.7]{\scriptsize$a_2$};
\draw[very thick] (0,1) --(0.88,-0.5) node[midway,sloped]{$>$}  ;
\draw (0.3,0.75) node[scale=0.7]{\scriptsize$c_1$};
\draw (0.7,0) node[scale=0.7]{\scriptsize$c_2$};
\draw (-0.88,-0.7) node[scale=0.7]{\scriptsize$v_2$};
\draw (0.88,-0.7) node[scale=0.7]{\scriptsize$v_3$};
\draw (0,1.2) node[scale=0.7]{\scriptsize$v_1$};
\draw (-0.1,0.1) node[scale=0.7]{\scriptsize$v_4$};
\draw (-0.88,-0.5) node{$\bullet$};
\draw (0.88,-0.5) node{$\bullet$};
\draw (0,0) node{$\bullet$};
\draw (0,1) node{$\bullet$};
\end{tikzpicture}}

\begin{document}
 \title{A three-manifold invariant from graph configurations}
 \author{Yohan Mandin--Hublé}
 \maketitle

 \begin{abstract}
The logarithm of the Kontsevich-Kuperberg-Thurston invariant counts embeddings of connected trivalent graphs in an oriented rational homology sphere, using integrals on configuration spaces of points in the given manifold. It is a universal finite type invariant of oriented rational homology spheres. The exponential of this invariant is often called the perturbative expansion of the Chern-Simons theory. In this article, we give an independent original definition of the degree two part of the logarithm of the Kontsevich-Kuperberg-Thurston invariant appropriate for concrete computations. This article can also serve as an introduction to the general definition of the logarithm of the Kontsevich-Kuperberg-Thurston invariant.

\end{abstract}

\vspace{1em}

\vspace{0.5em}
 
 \tableofcontents

\section{Introduction}

All manifolds will be oriented. In this article, a \emph{rational homology sphere} is a $3$-dimensional smooth manifold with the same homology with rational coefficients as $S^3$. For a homology sphere $M$, the logarithm $\zkkt(M)$ of the Kontsevich-Kuperberg-Thurston invariant $\Zkkt(M)$ counts embeddings of connected trivalent graphs in $M$, using integrals on configuration spaces of points in $M$. The $\Zkkt$ invariant is a universal finite type invariant of oriented rational homology spheres. It is valued in a quotient of the graded $\R$-vector space formally generated by trivalent graphs up to automorphisms. See Definitions \ref{trivalent} and \ref{def_jac}. The degree of a trivalent graph is half the number of its vertices. Kontsevich introduced the $\Zkkt$ invariant in \cite{ko}. Kuperberg and Thurston showed its universality among finite type invariants of integer homology spheres in \cite{kt}. Lescop further studied it in \cite{lesbookzv2}. In this article, we give a new independent definition of the degree two part $(\zkkt)_2$ of $\zkkt$ appropriate for concrete computations. This article may also serve as an introduction to the definition of $\zkkt$, which can be found in \cite{lesbookzv2}. 

The Le-Murakami-Ohtsuki (LMO) invariant is another universal finite type invariant of rational homology spheres. See \cite{lmo}. Moussard showed in \cite{moussardAGT} that the KKT and the LMO invariant carry the same information on rational homology spheres. However, the precise relationship between the two invariants remains to be established. It is known that the degree two parts of the LMO invariant and the KKT invariant coincide on integer homology $3$-spheres. Bar-Natan and Lawrence computed the LMO invariant on lens spaces in \cite{bnl}, but the values of $(\zkkt)_2$ on lens spaces are still unknown. In Section \ref{sec_LMO_KKT}, we show that the values of $(\zkkt)_2$ on lens spaces completely determine the difference between $(z_{KKT})_2$ and $(\zlmo)_2$, the degree two part of the logarithm of the LMO invariant.\footnote{In fact, knowing the values of $(\zkkt)_2$ on the lens spaces $L(p,1)$ for all prime numbers $p$ would be sufficient.} See Proposition \ref{prop_LMO_KKT}. Lescop obtained a surgery formula for $(\zkkt)_2$ involving these values. See \cite[Theorem $7.1$]{lezsurg}. The simple definition contained in the article opens a way to a complete determination of $(\zkkt)_2$. 

The article is organized as follows. In Section \ref{sec_main}, we state Theorem \ref{thm_secondcount}, which is the main result of this paper. It asserts that a certain sum $\lambda_2$ of integrals over the configuration space of four distinct points in $M$ is an invariant of $M$. The forms to be integrated are associated with the four edge-oriented trivalent graphs of Figure \ref{graphs}. In Section \ref{sec_compare}, we recall Lescop's definition of $(\zkkt)_2$, see Theorem \ref{thm_Lcount}. We recover Theorem \ref{thm_Lcount} from our Theorem \ref{thm_secondcount} in Remark \ref{rque_Lsecond}. In this remark, we also show that the invariant $\lambda_2$ in Theorem \ref{thm_secondcount} determines $(\zkkt)_2$. In Section \ref{sec_compactification}, we provide some useful details about the compactifications of configuration spaces of points in $M$ used in these constructions. Section \ref{sec_invariance} is then devoted to the proof of Theorem \ref{thm_secondcount} up to a proposition proved in Section \ref{sec_one}. Section \ref{sec_app} gives more details about expected applications of our theorem.

\subsection{Main result} \label{sec_main}
 Let us introduce the objects and give the necessary definitions to state Theorem \ref{thm_secondcount}. Let $M$ be an oriented rational homology sphere. Let $\infty$ be a point in $M$. Let $\check{M}$ be $M\setminus\{\infty\}$. Let
$$\Delta_{\check{M}}=\{(x,x), x\in\check{M}\}$$
be the diagonal of $\check{M}^2$. Let
$$\check{C}_2(\check{M})=\check{M}^2\setminus\Delta_{\check{M}}$$ be the space of injections of two points into $\check{M}$. We work with a particular compactification of $\check{C}_2(\check{M})$, which we denote by $C_2(M)$ (see Definition \ref{def_c2_compact}). It is a smooth manifold with boundary and ridges obtained by a natural compactification process for spaces of injections of finite sets into $\check{M}$. See Section \ref{sec_compactification}. There is a natural inclusion of $U\check{M}$ into the boundary $\partial C_2(M)$ of $C_2(M)$, where $U\check{M}$ is the unit tangent bundle to $\check{M}$: for any immersion $\gamma \colon ]-1,1[ \to \check{M}$, the limit at $(t=0)$ of $((\gamma(0),\gamma(t)) \in \check{C}_2(\check{M}))_{t\in ]0,1[}$ exists in $C_2(M)$. It belongs to the part $U\check{M}$ of $\partial C_2(M)$. It is the direction of $\gamma^{\prime}(0)$ at
$\gamma(0)$.

In the special case when $M=S^3=\R^3\cup \{\infty\}$, we define the Gauss map
$$ 
\begin{array}{llcl}
 G_{S^3} :& \check{C}_2(\R^3)   & \rightarrow & S^2 \\
  &  (x,y) & \mapsto & \frac{y-x}{||y-x||}, 
\end{array}
$$
where $||.||$ is the usual Euclidean norm on $\R^3$ and $S^2$ is the unit sphere of $\R^3$. The map $G_{S^3}$ smoothly extends to $C_2(S^3)$. For a general $M$, we choose an identification between a neighbourhood of $\infty$ in $M$ and a neighbourhood of $\infty$ in $S^3$. Then $G_{S^3}$ induces a Gauss map $G_M$ on $\partial C_2(M)\setminus U\check{M}$. See Definition \ref{def_gauss} and the discussion above it. Using the map $G_M$, we now define special closed two-forms on $C_2(M)$ as in \cite[Definition 3.11]{lesbookzv2}.
\begin{defi}\label{def_propform}
A \emph{propagating form} of $M$ is a closed two-form $\omega$ on $C_2(M)$ such that:
$$\omega_{|\partial C_2(M) \setminus U\check{M}}=G_M^*(\omega_{S^2}),$$
for a closed two-form $\omega_{S^2}$ on $S^2$ with
$$\int_{S^2}\omega_{S^2} =1.$$
\end{defi}
As $G_{S^3}$ is globally defined on $C_2(S^3)$, a typical example of a propagating form of $S^3$ is $G_{S^3}^*(\omega_{S^2})$. Propagating forms exist for any rational homology sphere for homological reasons. See \cite[Chapter $3$, Section $3$]{lesbookzv2} for a proof.

\begin{figure}
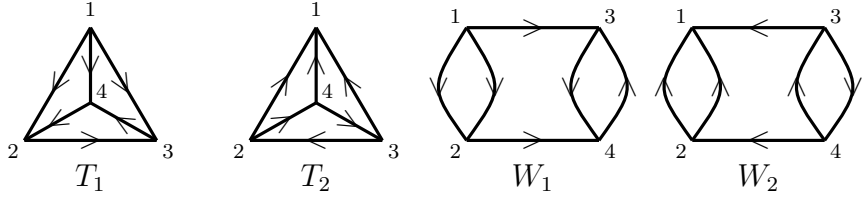

\centering
    \fourgraphs
    \caption{Four graphs.}
    \label{graphs}
\end{figure}

Let $\check{C}_4(\check{M})$ be the space of injections of $\{1,2,3,4\}$ into $\check{M}$. We orient it as a subset of $M^4$, which is oriented as an ordered product of oriented manifolds. Consider the four edge-oriented trivalent graphs with four vertices of Figure $\ref{graphs}$. Let $\Gamma$ be in $\{T_1,T_2,W_1,W_2\}$. Let $E(\Gamma)$ be the set of edges of $\Gamma$. For a pair $\{u,v\}$ of distinct elements of $\{1,2,3,4\}$, we let $e_{uv}(\Gamma)\subset E(\Gamma)$ be the set of edges of $\Gamma$ that go from $u$ to $v$. When the cardinality of $e_{uv}(\Gamma)$ is one and the context is clear, we will simply write $e_{uv}$ for the edge of $\Gamma$ going from $u$ to $v$. We now define a map $p_e$ from $\check{C}_4(\check{M})$ to $\check{C}_2(\check{M})$ for each $e\in E(\Gamma)$.
\begin{defi}\label{def_proj}
Let $\{u,v\}$ be a pair of distinct elements of $\{1,2,3,4\}$. Let $e$ be in $e_{uv}(\Gamma)$ and let $c\colon \{1,2,3,4\} \rightarrow \check{M}$ be an injection. We define:
$$p_e(c)=(c(u),c(v)).$$
\end{defi}
In the next definition, we introduce the above-mentioned integrals over $\check{C}_4(\check{M})$. When $X$ is a set of cardinality $6$, we define $\mathfrak{S}_6(X)$ to be the set of bijections from $X$ to $\pert$.
\begin{defi}\label{def_int}
 Let $j$ be in $\mathfrak{S}_6(E(\Gamma))$. For a family $(\omega_i)_{i\in\pert}$ of six propagating forms of $M$, we define the following integral:
$$I(M,(\omega_i)_i,(\Gamma,j))= \int_{\check{C}_4(\check{M})}  \bigwedge_{e \in E(\Gamma)} p_e^*(\omega_{j(e)}).$$
\end{defi}
These integrals converge because the maps $p_e$ smoothly extend to a natural compactification $C_4(M)$ of $\check{C}_4(\check{M})$. See the discussion at the end of Section \ref{sec_compactification}. We can now state our main result on the topological invariance of a count of configurations of the four graphs of Figure \ref{graphs} in $M$.
\begin{thm}\label{thm_secondcount}
Let $(\omega_i)_{i\in\pert}$ be a family of six propagating forms of $M$. The real number
$$\lambda_2(M,(\omega_i)_i)=\frac{1}{2^4 \times 6!}\left(\frac{1}{3}\sum_{\substack{\Gamma \in\{T_1,T_2\}\\ j \in \mathfrak{S}_6(E(\Gamma))}}I(M,(\omega_i)_i,(\Gamma,j))+\sum_{\substack{\Gamma \in\{W_1,W_2\}\\ j \in \mathfrak{S}_6(E(\Gamma))}} I(M,(\omega_i)_i,(\Gamma,j))\right)$$
does not depend on the chosen family $(\omega_i)_{i\in\pert}$. It only depends on $M$ up to orientation-preserving diffeomorphism. We denote it by $\lambda_2(M)$.
\end{thm}

The invariant $(\zkkt)_2$ is valued in the one-dimensional real vector space generated by $\left[\tetraeq\right]$. Let us define a numerical invariant $\tilde{\lambda}_2$ of rational homology spheres by the formula $$(\zkkt)_2 = \tilde{\lambda}_2 \left[\tetraeq\right].$$ We will prove that $\tilde{\lambda_2}=\lambda_2$. See Remark \ref{rque_Lsecond}. Our new definition should be more practical for computing $\tilde{\lambda_2}$ on specific manifolds. See Remark \ref{rque_simplier}. To compute integrals on configuration spaces as in Definition \ref{def_int}, one can try to discretize them. This amounts to replacing the integrals by a count of signed intersection points of transversely intersecting propagating chains defined as follows (see \cite[Chapter $11$]{lesbookzv2}). A \emph{propagating chain} of $M$ is a relative four-cycle $P$ of $(C_2(M),\partial C_2(M))$ with rational coefficients, transverse to $\partial C_2(M)$, and such that:
$$P\cap (\partial C_2(M) \setminus U\check{M}) = G^{-1}_M(\{a\}),$$
where $a$ is a point in $S^2$. A typical example of a propagating chain of $S^3$ is $G_{S^3}^{-1}(\{a\})$. This propagating chain has a simple geometric interpretation: a pair $(x,y)$ of distinct points in $\R^3$ belongs to $G_{S^3}^{-1}(\{a\})$ if and only if $y$ stands on the half-line starting at x and directed by $a$. Propagating chains are dual to propagating forms. They exist for any rational homology sphere because the associated configuration space $C_2(M)$ has the same rational homology as the sphere $S^2$. 
For concrete computations, one can choose a set of six propagating chains $(P_i)_{i\in \pert}$ and replace the integrals $I(M,(\omega_i)_i,(\Gamma,j))$ of Definition \ref{def_int} by the algebraic intersections of the $p_e^{-1}(P_{j(e)})$ for $e\in E(\Gamma)$ as in \cite[Lemma 11.7]{lesbookzv2}. These quantities make sense when the $(p_e^{-1}(P_{j(e)}))_{e\in E(\Gamma)}$ intersect transversally.

\begin{rque} \label{rque_propMorse}
Kuperberg and Lescop constructed specific propagating chains from the data of a Morse function on $\check{M}$ whose critical points are of index one or two, see \cite{lesHC}.\footnote{See also \cite{watanabeMorse} for an alternative construction by Watanabe.} Let $\phi$ be a Morse-Smale flow $\phi \colon \R\times\check{M} \to \check{M}$ associated to such a function. Let $(\A_i)_{i \in \genre}$ denote the collection of the ascending manifolds of the index 1 critical points and let $(\B_j)_{j \in \genre}$ denote the collection of the descending manifolds of the index 2 critical points. A propagating chain $P_{\phi}$ associated to $\phi$ is the sum of the closure of a flow part $\{(x,\phi(t,x));t\in \left]0,\infty\right[\}$ and a combination of closures of $\bigl(\B_j \setminus (\A_i \cap \B_j)\bigr) \times \A_i$. Using these propagators, Lescop computed the degree one part of the $\zkkt$ invariant in terms of the combinatorics of a Heegaard diagram of $M$.

A special feature of our graphs $T_1,T_2,W_1,W_2$ is that they all contain a cycle of oriented edges \cycletetra or \cycletheta. Apply our formula with a propagating form $\omega_{\phi}$ supported in a small neighborhood of a Morse propagating chain $P_{\phi}$ as above. Then the four-point configurations in the support of $\wedge_{e \in E(\Gamma)} p_e^*(\omega_{\phi})$ must map the vertices of an oriented cycle to a small neighborhood of the union of the $\A_i \cap \B_j$. The existence of oriented cycles induces a similar strong constraint on the intersection associated to a generic family of six Morse propagating chains obtained by small perturbations of $P_{\phi}$. We are computing  $\lambda_2$ from Heegaard diagrams using these constraints in a work in progress.
\end{rque}

\subsection{The degree two part of \texorpdfstring{$\zkkt$}{the KKT invariant}} \label{sec_compare}

We review Lescop's definition of $\tilde{\lambda}_2$ in Theorem \ref{thm_Lcount}. In Remark \ref{rque_simplier}, we explain to what extent our result gives a definition of $\tilde{\lambda}_2$ simpler than the original one.

Let $\T$ (resp. $\W$) denote the set of graphs obtained from $T_1$ (resp. from $W_1$) by changing in any possible way the orientations of the edges. The following theorem is a consequence of Theorem $7.19$ of \cite{lesbookzv2}, which gives a definition of the whole $\zkkt$.

\begin{thm}\label{thm_Lcount}
Let $(\omega_i)_{i\in\pert}$ be a family of six propagating forms of $M$. The real number
$$\tilde{\lambda}_2\bigl(M,(\omega_i)_i\bigr)=\frac{1}{2^6 \times 6!}\left(\frac{1}{24}\sum_{\substack{\Gamma \in\T \\ j \in \mathfrak{S}_6(E(\Gamma))}}I\bigl(M,(\omega_i)_i,(\Gamma,j)\bigr)+\frac{1}{8}\sum_{\substack{\Gamma \in\W\\ j \in \mathfrak{S}_6(E(\Gamma))}} I\bigl(M,(\omega_i)_i,(\Gamma,j)\bigr)\right)$$
does not depend on the chosen family $(\omega_i)_{i\in\pert}$. It only depends on $M$ up to orientation-preserving diffeomorphism. We denote it by $\tilde{\lambda}_2(M)$.
\end{thm}

\begin{rque}\label{rque_Lsecond}
 Theorem \ref{thm_Lcount} can be seen as a consequence of Theorem \ref{thm_secondcount}. Let $\iota \colon C_2(M) \rightarrow C_2(M)$ be the smooth extension of the involution of $\check{C}_2(\check{M})$ that exchanges the two coordinates. Note that if $\omega$ is a propagating form of $M$, then $\frac{1}{2}(\omega-\iota^*(\omega))$ is also a propagating form of $M$. Let $(\omega_i)_{i\in\pert}$ be a family of six propagating forms of $M$. We have the following equality
 $$\tilde{\lambda}_2\bigl(M,(\omega_i)_i\bigr)=\lambda_2\left(M,\left(\frac{1}{2}(\omega_i - \iota^*(\omega_i))\right)_{i\in\pert}\right) = \lambda_2(M).$$
 Thus Theorem \ref{thm_Lcount} is proved. Moreover, we get $\tilde{\lambda}_2(M)=\lambda_2(M)$.
\end{rque}

\begin{rque}\label{rque_simplier}

The formula of Theorem \ref{thm_Lcount} involves graphs that do not contain any cycle of oriented edges. In Remark \ref{rque_propMorse} we used the cycles of oriented edges in the graphs $T_1,T_2,W_1,W_2$ to get constraints on the supports of the forms to be integrated. We do not have such a constraint using Lescop's definition. Moreover, Lescop's definition involves more integrals than our formula.

Recall from the end of Section \ref{sec_main} that one can compute $\lambda_2$ using propagating chains instead of propagating forms. Both Lescop's definition and our formula allow one to use six distinct propagating chains. This freedom is necessary to get transversality, as we show in Section \ref{sec_transversality}.

\end{rque}


\begin{rque}\label{rque_para}
    With the additional choice of an asymptotically standard trivialisation $\tau \colon \check{M}\times \R^3 \rightarrow T\check{M}$ of the tangent bundle of $\check{M}$, one can define a Gauss map on $\partial C_2(M)$ (see \cite[Definition $3.6$, Proposition $3.7$]{lesbookzv2}). In fact, the invariant $\zkkt$ is first defined for a pair $(M,\tau)$. The behaviour of $\zkkt(M,\tau)$ when $\tau$ varies is determined by an element of the target space of the $\zkkt$ invariant called the \emph{beta anomaly}. See \cite[Chapter 10.6]{lesbookzv2}. The fact that the even degree part of the beta anomaly is zero implies that the even degree part of $\zkkt(M,\tau)$ does not depend on $\tau$. Alternatively, one can forget the parallelisation $\tau$ and construct the even degree part of $\zkkt$ as a function of $M$ in a direct way, as we did in Theorem \ref{thm_Lcount}.
\end{rque}

\subsection{Compactifications of configuration spaces} \label{sec_compactification}
Below, we give some details and explanations about the compactifications of configuration spaces of points in $M$ that we have used so far. We review the construction of $C_2(M)$ and give its important properties in Definition \ref{def_c2_compact} and Proposition \ref{bordcdeu}. In Definition \ref{def_gauss}, we describe the Gauss map $G_M$ on $\partial C_2(M)\setminus U\check{M}$. 

Let us start by some conventions regarding orientations. We orient any product of oriented manifolds, written from left to right, with the concatenation of the orientations of the manifolds in the same order. We orient the boundary of a manifold oriented by $o$ with the outward normal convention, that is with $o'$ such that $(n_\text{ext},o')=o$, where $n_\text{ext}$ is the outward normal. For an oriented manifold $N$, the same manifold equipped with the opposite orientation is denoted by $-N$.

Let $A$ be a smooth manifold. Let $TA$ be the tangent bundle to $A$. The \emph{unit tangent bundle} to $A$ is the fiber bundle over $A$ whose fiber over $x\in A$ is the quotient $(T_xA\setminus \{0\})/\R_+^*$, where $\R_+^*$ acts by scalar multiplication. We denote it by $UA$. If $B$ is a smooth submanifold of $A$, the \emph{unit normal bundle} to $B$ in $A$ is the bundle over $B$ defined as:
$$((TA/TB)\setminus \{0\})/\R_+^*,$$
where $\R_+^*$ acts by scalar multiplication. \emph{Blowing up} $B$ in $A$ in differential topology is an operation precisely defined in \cite[Definition $3.1$]{lesbookzv2}. The result of this operation is a smooth manifold with boundary. It is homeomorphic to the complement of an open tubular neighbourhood of $B$ in $A$ and is denoted by $Bl(A,B)$. It is equipped with a canonical smooth projection $Bl(A,B)\rightarrow A$. The restriction of this projection to the preimage of $A\setminus B$ is a diffeomorphism. The preimage of $B$ is naturally diffeomorphic to the unit normal bundle of $B$ in $A$. Therefore, informally speaking, blowing up $B$ in $A$ amounts to replacing $B$ with its unit normal bundle in $A$. 

Let $M$ be an oriented rational homology sphere. Let $\infty$ be a point in $M$. Let $\check{M}$ be $M\setminus \{\infty\}$. A first example of blow-up is the manifold $BL(M,\{\infty\})$. It is a compactification of $\check{M}$. As $\check{M}$ is the space of one-point configurations in $\check{M}$, we let $C_1(M)$ denote $BL(M,\{\infty\})$. In general, the blow-up operation is defined for manifolds with boundary and ridges and proper submanifolds transverse to the ridges. These notions allow one to perform a sequence of blow-ups on transversely intersecting submanifolds. This is used in the following construction of $C_2(M)$, in which a sequence of four blow-ups is performed.
\begin{defi}\label{def_c2_compact}
Let $C_2(M)$ be the oriented compact $6$-manifold with boundary and ridges obtained by first blowing up $\{(\infty,\infty)\}$ in $M\times M$, and then blowing up the closures of the pull-backs of $\check{M} \times \{\infty\}$, $\{\infty\} \times \check{M}$, and $\Delta_{\check{M}}=\{(x,x), x \in \check{M}\}$ by the natural smooth projection $p_{M^2} : Bl(M^2,(\infty,\infty)) \rightarrow M^2$.
\end{defi}

As a result of the last blow-up of this construction, the unit normal bundle to $\Delta_{\check{M}}$ in $\check{M}^2$ is a part of the boundary of $C_2(M)$. It is naturally identified with $U\check{M}$ in the following way:
$$
\begin{array}{lcl}
((T\check{M}^2/T\Delta_{\check{M}})\setminus\{0\})/\R_+^* &\overset{\sim}{\rightarrow}& U\check{M} \\
\left[(x,y)\right] & \mapsto & \left[y-x\right].
\end{array}
$$
We recall the following important properties of $C_2(M)$ without proofs. We refer to \cite[Chapter $3$, Section $2$]{lesbookzv2}.
\begin{prop} \label{bordcdeu}
The manifold $C_2(M)$ is a smooth compactification of $\check{C}_2(\check{M})$. It has the same rational homology as the $2$-sphere $S^2$. The boundary of $C_2(M)$ is:
$$\partial C_2(M) = p_{M^2}^{-1}(\{(\infty,\infty)\})\cup \left(S^2_\infty(M) \times \check{M}\right) \cup \left(-\check{M}\times S^2_\infty(M)\right) \cup U\check{M},$$
where $S^2_\infty(M)$ is the unit normal bundle to $\infty$ in $M$, oriented as the boundary of $C_1(M)$. Let $$\Bigl(G_{S^3}\colon (x,y) \mapsto \frac{y-x}{||y-x||}\Bigr)$$ be the \emph{Gauss map} from $\check{C}_2(\R^3)$ to $S^2$. Then $G_{S^3}$ has a smooth extension to $C_2(S^3)$.
\end{prop} 
We can now be more precise about the definition of the Gauss map $G_M$. Let us first study the case where $M=S^3$. We define a smooth map $p_\infty$ from $S^2_\infty(S^3)$ to $S^2$ as follows. For $v\in S^2_\infty(S^3)$, the element $p_\infty(v) \in S^2$ is the point such that $v$ is the limit in $C_1(S^3)$ of the sequence $(np_\infty(v))_{n\in\N}$. 
We use $p_\infty$ to describe $G_{S^3}$ on $\left(S^2_\infty(S^3) \times \R^3\right) \cup -\left(\R^3\times S^2_\infty(S^3)\right)$ as follows:
$$
 G_{S^3} = 
 \left\{
 \begin{array}{lcl}
 \iota_{S^2} \circ p_\infty  \circ p_1 &\text{ on }& S^2_\infty(S^3) \times \R^3 \\
 p_\infty \circ p_2 &\text{ on }& \R^3\times S^2_\infty(S^3)
\end{array}
\right.,$$
where $\iota_{S^2}$ is the antipodal involution of $S^2$, and $p_1$ and $p_2$ are the projections on the first and second factors respectively. Note that we have:
$$\partial C_2(M)\setminus U\check{M}=p_{M^2}^{-1}(\{(\infty,\infty)\})\cup \left(S^2_\infty(M) \times \check{M}\right) \cup \left(-\check{M}\times S^2_\infty(M)\right).$$
Let $\mathring{B}_{1,\infty}$ be the complement in $S^3$ of the closed ball of radius $1$ in $\R^3$. Choose an orientation-preserving diffeomorphism between an open neighbourhood of $\infty$ in $M$ and $\mathring{B}_{1,\infty}$. This diffeomorphism allows us to identify $p_{M^2}^{-1}(\{(\infty,\infty)\})$ with $p_{(S^3)^2}^{-1}(\{(\infty,\infty)\})$, and $S^2_\infty(M)$ with $S^2_\infty(S^3)$. We can write:
$$\partial C_2(M)\setminus U\check{M}=p_{(S^3)^2}^{-1}(\{(\infty,\infty)\})\cup \left(S^2_\infty(S^3) \times \check{M}\right) \cup \left(-\check{M}\times S^2_\infty(S^3)\right).$$
Using the identifications, we can define the Gauss map $G_M$ on $\partial C_2(M)\setminus U\check{M}$.

 \begin{defi}\label{def_gauss}
The \emph{Gauss map} $G_M$ from  $\partial C_2(M)\setminus U\check{M}$ to $S^2$ is defined as follows:
$$
 G_M = 
 \left\{
 \begin{array}{lcl}
 G_{S^3} &\text{ on }& p_{(S^3)^2}^{-1}(\{(\infty,\infty)\}) \\
 \iota_{S^2} \circ p_\infty  \circ p_1 &\text{ on }& S^2_\infty(S^3) \times \check{M} \\
 p_\infty \circ p_2 &\text{ on }& \check{M}\times S^2_\infty(S^3)
\end{array}
\right..$$
\end{defi}

There is a natural \say{Fulton-MacPherson} type compactification process for the sets $\check{C}_{X}(\check{M})$ of injections of $X$ into $\check{M}$, where $X$ is a finite set. We already described this compactification process when $X=\{1,2\}$, see Definition \ref{def_c2_compact}. For each $M$ as above, there is a contravariant functor from the category whose objects are finite sets and whose maps are injections to the category whose objects are smooth manifolds with boundary and ridges and whose maps are smooth maps. The functor maps a finite set $X$ to a compactification $C_X(M)$ of the set $\check{C}_{X}(\check{M})$ of injections of $X$ into $\check{M}$. If $f\colon Y\rightarrow X$ is an injection, the functor sends $f$ to a smooth map $p_f \colon C_{X}(M) \rightarrow C_{Y}(M)$, which extends the restriction map from $\check{C}_{X}(\check{M})$ to $\check{C}_{Y}(\check{M})$. associated with $f$. We let $C_4(M)$ be the space $C_{\{1,2,3,4\}}(M)$. A detailed exposition of this construction can be found in \cite[Chapter $8$]{lesbookzv2}. We describe the open codimension one faces of $C_4(M)$ in Definitions \ref{def_faceinf}, \ref{def_facepasinf}, and in Proposition \ref{prop_listface}.

Let $\Gamma$ be in $\{T_1,T_2,W_1,W_2\}$. Let $e$ be in $e_{uv}(\Gamma)$. We again denote by $p_e$ the natural smooth extension from $C_4(M)$ to $C_2(M)$ of the map $p_e$ of Definition \ref{def_proj}. Note that the former $p_e$ was exactly the restriction map from $\check{C}_4(M)$ to $\check{C}_2(M)$ associated with the following map from $\{1,2\}$ to $\{1,2,3,4\}$:
$$
\begin{array}{lcl}
\{1,2\} & \rightarrow & \{1,2,3,4\} \\
1 & \mapsto & u \\
2 & \mapsto & v.
\end{array}
$$
Let $(\omega_i)_i$ be a family of six propagating forms. Using the extended maps $p_e$, we have:
$$I(M,(\omega_i)_i,(\Gamma,j))= \int_{C_4(M)}  \bigwedge_{e \in E(\Gamma)} p_e^*(w_{j(e)}).$$
The only difference with Definition \ref{def_int} is that the space over which the forms are integrated is now compact. It is an important difference, since it ensures that the integrals converge.

\subsection{Acknowledgements}
I thank my advisor, Christine Lescop, for her careful reading and for suggestions.

\section{Proof of invariance}\label{sec_invariance}

\subsection{Sketch of proof}\label{sec_sketch}

In this section we show the invariance of $\lambda_2$ with respect to the change of propagating forms. Let $(\omega_i^0)_i$ and $(\omega_i^1)_i$ be two sets of propagating forms, respectively associated with sets $(\omega_{i,S^2}^0)_i$ and $(\omega_{i,S^2}^1)_i$ of volume-one forms on $S^2$ as in Definition \ref{def_propform}.

First, for each $i$, let $\eta_i$ be a one-form on $S^2$ such that $\omega_{i,S^2}^1-\omega_{i,S^2}^0=d\eta_i$. We define a closed two-form $\omega_{i,S^2}$ on $\left[0,1\right]\times S^2$ by: 
$$\omega_{i,S^2}^t= p_2^*(\omega_{i,S^2}^0) + d(t\eta_i),$$
where $p_2$ is the projection on the second factor and $t$ is the projection on the first factor $[0,1]$. Then $(\mathbf{1}_{\left[0,1\right]} \times G_M)^*(\omega_{i,S^2}^t)$ is a closed two-form on $$\left[0,1\right] \times \left(p_{M^2}^{-1}(\{(\infty,\infty)\})\cup \left(S^2_\infty(M) \times \check{M}\right) \cup \left(-\check{M}\times S^2_\infty(M)\right)\right) \subset [0,1] \times\partial C_2(M).$$ We extend it to a closed two-form on $\left[0,1\right] \times C_2(M)$ denoted by $\omega_i$, whose restriction to $\{0\}\times C_2(M)$ (resp. to $\{1\}\times C_2(M)$) is $\omega_i^0$ (resp. $\omega_i^1$). It is possible as the relevant relative cohomology groups are trivial, see \cite[Lemma $9.1$]{lesbookzv2}.

Let $\Gamma$ be in $\{T_1,T_2,W_1,W_2\}$ and $j$ be in $\bijlab(E(\Gamma))$. We now apply the Stokes theorem on $\left[0,1\right]\times C_{4}(M)$ to the following closed form:
$$\bigwedge_{e\in E(\Gamma)} (\mathbf{1}_{\left[0,1\right]} \times p_e)^*(\omega_{j(e)}).$$
The boundary of $C_4(M)$ admits a stratification. Let $(C_4(M))_1$ be the set of open codimension one faces of $C_4(M)$. We obtain:
$$
I(M,(\omega_i^1)_i,(\Gamma,j))-  I(M,(\omega_i^0)_i,(\Gamma,j))
= \sum_{F\in (C_4(M))_1} I(M,(\omega_i)_i,\Gamma,j,F)
$$
with
$$I(M,(\omega_i),\Gamma,j,F)= \int_{\left[0,1\right] \times F} \bigwedge_{e\in E(\Gamma)} (\mathbf{1}_{\left[0,1\right]} \times p_e)^*(\omega_{j(e)}).$$
In what follows, the integrals $I(M,(\omega_i),\Gamma,j,F)$ are denoted by $I(\Gamma,j,F)$. With this notation, Theorem \ref{thm_secondcount} is equivalent to the following proposition.
\begin{prop}\label{prop_main}
We have:
$$\sum_{\substack{\Gamma\in \{T_1,T_2\}\\j\in\bijlab(E(\Gamma))}} \sum_{F\in (C_4(M))_1} I\left(\Gamma,j,F\right) + 3\sum_{\substack{\Gamma\in \{W_1,W_2\}\\j\in\bijlab(E(\Gamma))}} \sum_{F\in (C_4(M))_1} I\left(\Gamma,j,F\right)  = 0.$$
\end{prop}
To prove Proposition \ref{prop_main} we first describe $(C_4(M))_1$ and make a list of the terms that appear in the above equality. We then construct an appropriate partition of the set of terms and show that the sum of terms in each element of the partition is $0$. The elements of this partition have cardinality $1$, $2$ or $3$. The principles of the proofs of those various cancellations are not new. They were used by Bott and Taubes in \cite{botttaubes} and by Kontsevich in \cite{ko} for example. See also \cite[Chapter $9$]{lesbookzv2}.

\subsection{List of faces}
In this section, we enumerate the open codimension one faces of $C_4(M)$. We refer to \cite[Chapter $8$]{lesbookzv2} for a study of compactifications of configuration spaces. The next definition describes the faces where some points go to infinity. When $X$ is a finite set, we denote the cardinality of $X$ by $|X|$.
\begin{defi} \label{def_faceinf}
 Let $A$ be a subset of $\{1,2,3,4\}$ such that $|A|\geq 1$. Let $\check{S}(T_\infty M,A)$ be the space of injections of $A$ into $\R^3\setminus\{0\}$ up to dilation.\footnote{In this article, a \emph{dilation} is a map from $\R^3\setminus\{0\}$ to itself of the form $x\mapsto \mu x$ where $\mu$ is in $\R_+^*$.} We define $F(A,\infty)$ to be the product $\check{C}_{\{1,2,3,4\}\setminus A}(\check{M}) \times  \check{S}(T_\infty M,A)$.
\end{defi}
Let us describe the inclusion $F(\{1,2,3,4\},\infty)\subset \partial C_4(M)$. Note the equality $$F(\{1,2,3,4\},\infty)= \check{S}(T_\infty M,\{1,2,3,4\}).$$ 
Let $c$ be a point in $\check{S}(T_\infty M,\{1,2,3,4\})$. It can be represented by an injection $c'$ of $\{1,2,3,4\}$ into $\R^3\setminus\{0\}$ such that $||c'(u)||> 1$ for any $u\in\{1,2,3,4\}$. Recall that we have chosen an identification between $\mathring{B}_{1,\infty}$ and a neighbourhood of $\infty$ in $M$. The sequence $(nc')_{n\in\N^*}$ converges to a point in $\partial C_4(M)$, and we identify $c$ with this point. If $\Gamma$ is in $\{T_1,T_2,W_1,W_2\}$ and $e$ is in $e_{uv}(\Gamma)$, then $p_e(c)$ is the limit in $C_2(M)$ of the sequence $(nc'(u),nc'(v))_n$. By continuity, we have:
$$G_M\circ p_e (c)= \frac{c'(v)-c'(u)}{||c'(v)-c'(u)||}.$$

\begin{exemple} \label{ex_faceinfC2}
 Let us study those definitions in the case of $C_2(M)$. In a similar way to Definition \ref{def_faceinf}, we define, for $A$ a subset of $\{1,2\}$ such that $|A|\geq 1$:
 $$F(A,\infty)=\check{C}_{\{1,2\}\setminus A}(\check{M})\times \check{S}(T_\infty M,A).$$
 Recall the description of $\partial C_2(M)$ in Proposition \ref{bordcdeu}. 
 Then the faces $F(A,\infty)$ are precisely the open codimension faces of $C_2(M)$ where at least one point goes to $\infty$. Indeed, the space $\check{S}(T_\infty M,\{1\})$ is identified with $S^2_\infty(M)$ via the map $p_\infty$. The faces $F(\{1\},\infty)$ and $F(\{2\},\infty)$ are respectively identified with $S^2_{\infty}(M)\times \check{M}$, and $(-\check{M})\times S^2_{\infty}(M)$. The face $F(\{v_1,v_2\},\infty)=\check{S}(T_\infty M,\{v_1,v_2\})$ is identified with:
$$p_{M^2}^{-1}(\{(\infty,\infty)\})\setminus \overline{p_{M^2}^{-1}\Bigl((\{\infty\}\times \check{M}) \cup (\check{M}\times \{\infty\}) \cup \Delta_{\check{M}} \Bigr)}.$$
\end{exemple}
The next definition describes the open codimension faces of $C_4(M)$ where a bunch of points coincide. Together with the previously described faces where some points go to $\infty$, these are exactly all the open codimension one faces of $C_4(M)$. See Proposition \ref{prop_listface}.
\begin{defi}\label{def_facepasinf}
Let $A$ be a subset of $\{1,2,3,4\}$ such that $|A|\geq 2$, and $a\in A$. For $m\in \check{M}$, let $\check{S}_A(T_{m}M)$ be the space of injections of $A$ into $T_{m}M$ up to global translation and dilation. We define $F(A)$ as the fibered product over $\check{C}_{(\{1,2,3,4\}\setminus A) \cup\{a\}}(\check{M})$  whose fiber over a point $c$ is $\check{S}_A(T_{c(a)}M)$.
\end{defi}
Let us describe the inclusion $F(\{1,2,3,4\})\subset \partial C_4(M)$ when $M=S^3$. Note first that $F(\{1,2,3,4\})$ is the fibered product over $\R^3$ whose fiber over $m\in \R^3$ is $\check{S}_{\{1,2,3,4\}}(T_m\R^3)$. Let $c'\colon\{1,2,3,4\}\rightarrow T_m\R^3$ be a representative of a point $c$ in $\check{S}_{\{1,2,3,4\}}(T_m\R^3)$. We use the canonical identification between $T_m\R^3$ and $\R^3$ to see $c'$ as a map from $\{1,2,3,4\}$ to $\R^3$. Then the sequence $\left(\left(m+\frac{1}{n}c'(i)\right)_{i\in\{1,2,3,4\}}\right)_{n\in\N^*}$ converges to a point in $\partial C_4(M)$. We identify $(m,c)\in F(\{1,2,3,4\})$ with this point. The face $F(\{1,2,3,4\})$ is often referred to as the \emph{anomalous face}. 
\begin{exemple}\label{ex_facepasinf}
    Let us study Definition \ref{def_facepasinf} in the case of $C_2(M)$. The face $F(\{1,2\})$ is the fibered product over $\check{C}_{(\{1,2\}\setminus\{1,2\})\cup \{1\}}(\check{M})=\check{M}$ whose fiber over $m\in\check{M}$ is $\check{S}_{\{1,2\}}(T_mM)$. As $\check{S}_{\{1,2\}}(T_mM)$ is $UM_m$, the face $F(\{1,2\})$ is exactly $U\check{M}$. 
\end{exemple}
According to Proposition \ref{bordcdeu}, Examples \ref{ex_faceinfC2} and \ref{ex_facepasinf} describe all the open codimension one faces of $C_2(M)$. As in the case of $C_2(M)$, Definitions \ref{def_faceinf} and \ref{def_facepasinf} give an exhaustive list of the open codimension one faces of $C_4(M)$, as the following particular case of \cite[Proposition $8.18$]{lesbookzv2} shows.
\begin{prop}\label{prop_listface}
The open codimension one faces of $C_{4}(M)$ are the faces $F(A)$ for $|A|\geq 2$ and the faces $F(A,\infty)$ for $|A|\geq 1$.
\end{prop}

\subsection{Faces where some vertices tend to \texorpdfstring{$\infty$}{infinity}}\label{subsection_infinity}

Let $\Gamma$ be in $\{T_1,T_2,W_1,W_2\}$. Let $j$ be in $\bijlab(E(\Gamma))$. In this section, we show that $I(\Gamma,j,F(A,\infty))=0$ for any $A\subset \{1,2,3,4\}$ such that $|A|\geq 1$. This is the content of Lemmas \ref{lem_inf} and \ref{lem_inf_bis}.

The next lemma focuses on the face $F(\{1,2,3,4\},\infty)$. If $f$ is a smooth map between smooth manifolds, we denote the tangent map to $f$ at a point $c$ by $d_cf$.
\begin{lem} \label{lem_inf}
We have:
$$I(\Gamma,j,F(\{1,2,3,4\},\infty))=0.$$
 \end{lem}
 \begin{proof}
We identify $F(\{1,2,3,4\},\infty)$ with the set of injections $c\colon\{1,2,3,4\} \rightarrow \R^3\setminus\{0\}$ such that $||c(2)-c(1)||=1$. Let $c$ be such an injection. For $\lambda$ in a small neighbourhood of $0$ and injections $c'$ in a small neighbourhood of $c$, we define $\lambda \cdot c' \in F(\{1,2,3,4\},\infty)$ to be the injection:
$$v \mapsto \lambda + c'(v).$$
For every edge $e$, we have:
$$G_M \circ p_e(\lambda \cdot c)= G_M \circ p_e(c).$$ 
This follows from our description of $G_{M}\circ p_e$ on $F(\{1,2,3,4\})$. Thus the intersection of the kernels of the maps $d_c(G_M \circ p_e)$ for $e\in E(\Gamma)$ is non trivial on $[0,1]\times F(\{1,2,3,4\},\infty)$. Moreover, on $F(\{1,2,3,4\},\infty)$ we have:
$$(\mathbf{1}_{\left[0,1\right]} \times p_e)^* (\omega_{j(e)})= (\mathbf{1}_{\left[0,1\right]} \times (G_M \circ p_e))^* (\omega_{i,S^2}^t).$$
As the degree of this form is the dimension of  $[0,1]\times F(\{1,2,3,4\},\infty)$ and the intersection of the kernels of the maps $d_c(G_M \circ p_e)$ is non trivial, the right member of the previous equality vanishes at $c$. So it is identically zero on $[0,1]\times F(\{1,2,3,4\},\infty)$.
 \end{proof}
We now investigate the case of faces $F(A,\infty)$ where $A$ is different from $\{1,2,3,4\}$. For $c\in F(A,\infty)$, we write $c=(c_{\{1,2,3,4\}\setminus A},c_A)$. Let $e$ be in $e_{uv}(\Gamma)$. We now explicitly describe $p_e$ on the faces $F(A,\infty)$. Let $c$ be in $F(A,\infty)$. We distinguish four cases:

\begin{enumerate}
 \item if $u,v \in A$, then $p_e(c)=(c_A(u),c_A(v))\in p_{M^2}^{-1}(\{(\infty,\infty)\})$,
 \item if $u \in A$ and $v \notin A$, then $p_e(c)=(c_{A}(u),c_{\{1,2,3,4\}-A}(v)) \in S^2_\infty(M) \times \check{M}$,
\item if $u \notin A$ and $v \in A$, then $p_e(c)=(c_{\{1,2,3,4\}\setminus A}(u),c_{A}(v)) \in \check{M}\times S^2_\infty(M)$,
\item if $u, v \notin A$, then $p_e(c)=(c_{\{1,2,3,4\} \setminus A}(u),c_{\{1,2,3,4\}\setminus A}(v)) \in C_2(M)$.
 \end{enumerate}
 When $u\in A$ or $v\in A$, the image of the restriction of $p_e$ to $F(A,\infty)$ is included in the part of $\partial C_2(M)$ where the propagating forms are explicitly given as pull-backs by $G_M$ of $2$-forms on $S^2$. The next lemma follows from the description of $p_e$  on $F(A,\infty)$ and from Definition \ref{def_gauss}.
 \begin{lem}\label{lem_factorinf}
     Assume that $\{u,v\}\cap A \neq \emptyset$. The restriction of the map $G_M\circ p_e$ to $F(A,\infty)$ factors through the projection of $F(A,\infty)$ onto $\check{S}(T_\infty,A)$.
 \end{lem}
As it will be needed later, we let $E_A(\Gamma)$ denote the set of edges of $\Gamma$ with both ends in $A$.
\begin{lem} \label{lem_inf_bis}
   Let $A$ be a subset of $\{1,2,3,4\}$ such that $1\leq |A|\leq 3$. Then:
$$I\left(\Gamma,j, F(A,\infty)\right)=0.$$

 \end{lem}
 \begin{proof}
Let $E_C(\Gamma)$ be the set of edges between a vertex in $A$ and a vertex outside $A$. We prove that the following form vanishes on $F(A,\infty)$:
$$\bigwedge_{e \in E_A(\Gamma) \cup E_C(\Gamma)} (\mathbf{1}_{\left[0,1\right]} \times p_e)^* (\omega_{j(e)}).$$
This obviously implies that the form
$\bigwedge_{e \in E(\Gamma))} (\mathbf{1}_{\left[0,1\right]} \times p_e)^* (\omega_{j(e)})$
vanishes on $F(A,\infty)$, which in turn implies the lemma. Let $f$ be the projection of $F(A,\infty)$ on $\check{S}(T_\infty M,A)$. Let $e$ be in $E_C(\Gamma) \cup E_A(\Gamma)$. On $F(A,\infty)$ we have:
$$(\mathbf{1}_{\left[0,1\right]} \times p_e)^* (\omega_{j(e)})=(\mathbf{1}_{\left[0,1\right]} \times (G_M \circ p_e))^* (\omega_{j(e),S^2}).$$
Applying Lemma \ref{lem_factorinf} we can write:
$$G_M\circ p_e = p_e' \circ f,$$
where $p_e'$ is a map from to $\check{S}(T_\infty M,A)$ to $S^2$. With this notation, we have:
$$
\bigwedge_{e \in E_A(\Gamma) \cup E_C(\Gamma)} (\mathbf{1}_{\left[0,1\right]} \times p_e)^* (\omega_{j(e)})= (\indi \times f)^*\left(\bigwedge_{e \in E_A(\Gamma) \cup E_C(\Gamma)} (\mathbf{1}_{\left[0,1\right]} \times p_e')^* (\omega_{j(e),S^2})\right).
$$
The degree of the form $\bigwedge_{e \in E_A(\Gamma) \cup E_C(\Gamma)} (\mathbf{1}_{\left[0,1\right]} \times p_e')^* (\omega_{j(e),S^2})$ is  $2|E_C(\Gamma)|+2|E_A(\Gamma)|$. The dimension of $[0,1]\times\check{S}(T_\infty M,A)$ is $1+(3|A|-1)=3|A|$. Moreover, we have $3|A|=|E_C(\Gamma)|+2|E_A(\Gamma)|$. Since $E_C(\Gamma)$ is not empty, we get $2|E_C(\Gamma)|+2|E_A(\Gamma)|>3|A|$. So this form vanishes.
 \end{proof}

\subsection{Other degenerate faces}

In Section \ref{subsection_infinity} we encountered faces where the forms to be integrated actually vanish. For this reason, these faces are called \emph{degenerate faces}. Lemmas \ref{lem_degenerateW} and \ref{lem_degsansedge} provide other examples of degenerate faces.

Let $\Gamma$ be in $\{T_1,T_2,W_1,W_2\}$. Let $A$ be a subset of $\{1,2,3,4\}$ such that $|A|\geq 2$. Let $a$ be in $A$. For $c\in F(A)$, we write $c=(c_{(\{1,2,3,4\}-A) \cup\{a\}},c_A)$, meaning that $c_{(\{1,2,3,4\}\setminus A) \cup\{a\}}$ is the projection of $c$ on $\check{C}_{(\{1,2,3,4\} \setminus A) \cup\{a\}}(\check{M})$ and $c_A$ is a representative of $c\in\check{S}_A(T_{c(a)}M)$. Let $e$ be in $e_{uv}(\Gamma)$. We now explicitly describe $p_e$ on $F(A)$. Let $c$ be in $F(A)$. We distinguish four cases:
\begin{enumerate}
\item if $u,v \notin A$, then $p_e(c)=(c_{(\{1,2,3,4\}\setminus A) \cup\{a\}}(u),c_{(\{1,2,3,4\}\setminus A) \cup\{a\}}(v))$,
\item if $u,v \in A$, then $p_e(c)=(c_A(u),c_A(v))\in \partial C_2(M)$, 
\item if $u \in A$ and $v \notin A$, then $p_e(c)=(c_{(\{1,2,3,4\}\setminus A) \cup\{a\}}(a),c_{(\{1,2,3,4\}\setminus A) \cup\{a\}}(v)) \in \check{C}_2(\check{M})$,
\item if $u \notin A$ and $v \in A$, then $p_e(c)=(c_{(\{1,2,3,4\}\setminus A) \cup\{a\}}(u), c_{(\{1,2,3,4\}\setminus A) \cup\{a\}}(a)) \in \check{C}_2(\check{M})$.
\end{enumerate}
For $j\in \mathfrak{S}_6(E(\Gamma))$, the integral $I(\Gamma,j,F(A))$ will now be denoted by $I(\Gamma,j,A)$. 

\begin{lem}\label{lem_degenerateW}
Let $\Gamma$ be in $\{W_1,W_2\}$. Let $j$ be in $\bijlab(E(\Gamma))$. Let $A$ be a subset of $\{1,2,3,4\}$ such that $|A|=3$. Then:
$$I\left(\Gamma,j,A\right)=0.$$
\end{lem}
\begin{proof}
Let $u_1,u_2,u_3$ be the three vertices in $A$. Assume that they are as in Figure \ref{deg_W}. There are two edges between $u_2$ and $u_3$, and one edge between $u_1$ and $u_2$. Let $c$ be in $F(A)$. Let $c_A$ be a representative of the projection of $c$ on $\check{S}_A(T_{c(u_1)}M)$ such that $c_A(u_2)=0$. For $\lambda$ close to $1$, we define:
\begin{align*}
\lambda \cdot c_A \colon u_1 &\mapsto \lambda c_A(u_1) \\
u_2 &\mapsto 0 \\
u_3 &\mapsto c_A(u_3).
\end{align*}
Then the class of $\lambda \cdot c_A$ in $\check{S}_A(T_{c(u_1)}M)$ does not depend on the choice of $c_A$. Note moreover that if $\lambda\neq 1$, then $\lambda \cdot c_A$ and $c_A$ represent different elements of the configuration space. We define $\lambda \cdot c \in F(A)$ to be $(c_{(\{1,2,3,4\}\setminus A)\cup \{a\}},\lambda \cdot c_A)$. It is a non trivial action on $F(A)$. For all $e\in E(\Gamma)$ we have:
$$p_e(\lambda \cdot c)= p_e(c).$$
This follows from our description of the maps $p_e$ on $F(A)$, and from the fact that there is no edge connecting $u_1$ and $u_3$. Thus, taking the derivative at $\lambda=1$ shows that the intersection of the kernels of the maps $d_c p_e$ is non trivial on $[0,1]\times F(A)$. Moreover, the degree of the form $\bigwedge_{e\in E(\Gamma)} (\mathbf{1}_{\left[0,1\right]} \times p_e)^*(\omega_{j(e)})$ is the dimension of $[0,1]\times F(A)$. Therefore this form vanishes on $[0,1]\times F(A)$.
\end{proof}

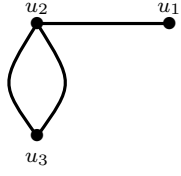
\begin{figure}
    \centering
\begin{tikzpicture}[scale=1]
\useasboundingbox (4.75,-0.8) rectangle (7,1.2);

\draw[very thick] (5,1) -- (6.76,1);
\draw[very thick] (5,-0.5) .. controls (5.5,0.2) ..(5,1);
\draw[very thick] (5,-0.5) .. controls (4.5,0.2) ..(5,1) ;

\draw (5,-0.8) node{\scriptsize$u_3$};
\draw (5,1.2) node{\scriptsize$u_2$};
\draw (6.76,1.2) node{\scriptsize$u_1$};

\draw (5,-0.5) node{$\bullet$};
\draw (5,1) node{$\bullet$};
\draw (6.76,1) node{$\bullet$}; 
\end{tikzpicture}
 \caption{Faces for $|A|=3$ and $\Gamma\in\{W_1,W_2\}$.}
    \label{deg_W}
\end{figure}

\begin{lem}\label{lem_degsansedge}
   Let $\Gamma$ be in $\{W_1,W_2\}$ and let $A$ be in $\{\{1,4\},\{2,3\}\}$. Let $j$ be in $\bijlab(E(\Gamma))$. We have:
   $$I(\Gamma,j,A)=0.$$
\end{lem}
\begin{proof}
    Let $a$ be a vertex in $A$. Let $f$ be the projection from $F(A)$ to $\check{C}_{(\{1,2,3,4\}-A)\cup \{a\}}(\check{M})$. Note that $E_A(\Gamma)$ is empty. Our description of the maps $p_e$ on $F(A)$ shows that all those maps factor through $f$. We write, for every $e\in E(\Gamma)$:
    $$p_e=p_e' \circ f.$$
  We have:
    $$\bigwedge_{e\in E(\Gamma)} (\mathbf{1}_{\left[0,1\right]} \times p_e)^*(\omega_{j(e)})=(\indi\times f)^*\left(\bigwedge_{e\in E(\Gamma)} (\mathbf{1}_{\left[0,1\right]} \times p_e')^*(\omega_{j(e)})\right).$$
As the dimension of $\check{C}_{(\{1,2,3,4\}\setminus A)\cup \{a\}}(\check{M})$ is smaller than the dimension of $F(A)$, we find that this form vanishes.
\end{proof}

\subsection{Reversing all the edges}

The next lemma explains how to cancel integrals over the anomalous face in our case. The same proof can be used in even degree. It does not apply in odd degree. We use the involution $\mathcal{I}$ that exchanges $T_1$ (resp. $W_1$) with $T_2$ (resp. $W_2$). Let $\Gamma$ be in $\{T_1,T_2,W_1,W_2\}$. Note that $\mathcal{I}$ acts by reversing the orientation of all edges. Let $\I \colon E(\Gamma)\rightarrow E(\mathcal{I}(\Gamma))$ be a bijection that sends $e_{uv}(\Gamma)$ to $e_{vu}(\I(\Gamma))$ for $u\neq v \in \{1,2,3,4\}$. 

\begin{lem} \label{lem_anomalous}
Let $\Gamma$ be in $\{T_1,W_1\}$. Let $j$ be in $\bijlab(E(\Gamma))$. We have:
$$
I(\Gamma,j,\{1,2,3,4\}) +I(\mathcal{I}(\Gamma),j\circ\I^{-1}, \{1,2,3,4\})=0.
$$
\end{lem}
\begin{proof}
The face $F(\{1,2,3,4\})$ fibers over $\check{M}$, and its fiber over $m\in\check{M}$ is $\check{S}_{\{1,2,3,4\}}(T_{m}M)$. Define:
$$
f \colon
\left\{
\begin{array}{lcl}
 F(\{1,2,3,4\}) &\rightarrow& F(\{1,2,3,4\}) \\
(m\in \check{M},c \colon \{1,2,3,4\} \rightarrow T_{m}M) & \mapsto &(m, -c)
\end{array}
\right.
$$
Then $f$ is orientation-reversing. Indeed, the fiber $\check{S}_{\{1,2,3,4\}}(T_{m}M)$ is homeomorphic to $S^8$ and $f$ acts as the antipodal involution on it. For $e\in E(\Gamma)$ we have $p_{\I(e)}\circ f=p_e$. Let $j'$ be $j \circ \I^{-1}$. We get:
\begin{align*}
I(\mathcal{I}(\Gamma),j',\{1,2,3,4\})
&= \int_{[0,1] \times F(\{1,2,3,4\})}  \bigwedge_{e \in E(\Gamma)} (\mathbf{1}_{\left[0,1\right]}\times p_{\I(e)})^*(\omega_{j'\circ\I(e)}) \\
&= - \int_{[0,1] \times F(\{1,2,3,4\})}  (\indi \times f)^*\left(\bigwedge_{e \in E(\Gamma)} (\mathbf{1}_{\left[0,1\right]}\times p_{\I(e)})^*(\omega_{j'\circ \mathcal{I}(e)})\right) \\
&=-I(\Gamma,j'\circ \mathcal{I},\{1,2,3,4\}) \\
&=-I(\Gamma,j,\{1,2,3,4\}).
\end{align*}
\end{proof}

\begin{nota}\label{not_intiota}
Recall the involution $\iota$ of $C_2(M)$ from Remark \ref{rque_Lsecond}. We define the following integrals:
$$I_\iota(\Gamma,j,A)=\int_{[0,1]\times F(A)} \bigwedge_{e\in E(\Gamma)} (\indi\times p_e)^*\bigl((\indi \times \iota)^*(\omega_{j(e)})\bigr).$$
\end{nota}
We will use the next lemma to deduce cancellations of terms involving $T_2$ and $W_2$ from similar cancellations involving $T_1$ and $W_1$.
\begin{lem}\label{lem_reverse}
Let $\Gamma\in \{T_1,W_1\}$. Let $j$ be in $\bijlab(E(\Gamma))$. Let $A$ is a subset of $\{1,2,3,4\}$ such that $|A|\geq 2$. We have:
   $$I(\I(\Gamma),j\circ \I^{-1},A)=I_{\iota}(\Gamma,j,A).$$
\end{lem}
\begin{proof}
Let $e$ be an edge of $\Gamma$. Then $p_{\I(e)}=\iota \circ p_e$.
Thus we get the required equality.
\end{proof}

\subsection{Double-edge faces}
We call \emph{double-edge faces} the pairs $(W_1,F(\{1,2\}))$, $(W_1,F(\{3,4\}))$, $(W_2,F(\{1,2\}))$, and $(W_2,F(\{3,4\}))$. In this section, we show that the contribution of double-edge faces cancel.
\begin{figure}[H]
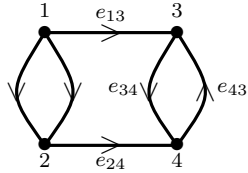

    \centering
    \doubleedge
    \caption{A name for the edges of $W_1$.}
    \label{doubleedge}
\end{figure}

\begin{lem} \label{lem_doubleedge1}
Let $j$ be in $\bijlab(E(W_1))$. Let $j'\in \bijlab(E(W_1))$ be the bijection only differing from $j$ by $j'(e_{34})=j(e_{43})$ and $j'(e_{43})=j(e_{34})$.
We have:
$$I(W_1,j,\{3,4\})=-I(W_1,j',\{3,4\}),$$
and:
$$I_\iota(W_1,j,\{3,4\})=-I_\iota(W_1,j',\{3,4\}),$$
with Notation \ref{not_intiota}.
\end{lem}
\begin{proof}
Let us prove the first equality. Define the following diffeomorphism:
$$
f \colon
\left\{
\begin{array}{lcl}
  F(\{3,4\}) &\rightarrow &F(\{3,4\}) \\
 (c_0\colon \{1,2,4\} \rightarrow \check{M},c_1 \colon \{3,4\}\rightarrow T_{c_0(4)}M) &\mapsto &(c_0,-c_1).
\end{array}
\right.
$$
The diffeomorphism $f$ is orientation-reversing. Indeed, the fiber $\check{S}_{\{3,4\}}(T_{c_0(4)}M)$ is homeomorphic to $S^2$ and $f$ acts on it as the antipodal involution. On $F(\{3,4\})$, we have:
\begin{align*}
 &p_{e_{43}}\circ f = p_{e_{34}}, \\
 &p_{e_{34}}\circ f = p_{e_{43}}, \\
 &p_e \circ f=p_e \text{ for } e\in E(\Gamma)-\{e_{43},e_{34}\}.
\end{align*}
This implies that for every edge $e$ we have: 
$$\bigl((\mathbf{1}_{[0,1]} \times p_e) \circ (\mathbf{1}_{[0,1]}\times f)\bigr)^*(\omega_{j'(e)})=(\mathbf{1}_{[0,1]} \times p_e)^*(\omega_{j(e)}),$$
 which proves the first equality lemma, using $\indi\times f$ as a change of variables. The same proof applies for the second equality.
\end{proof}

\begin{lem}\label{lem_doubleedge2}
Let $j$ be in $\bijlab(E(W_1))$. Let $j'\in\bijlab(E(W_1))$ be the bijection differing from $j$ by:
\begin{align*}
 j'(e_{24})&=j(e_{13}), \\
 j'(e_{13})&=j(e_{24}), \\
 j'(e_{43})&=j(e_{34}), \\
 j'(e_{34})&=j(e_{43}).
 \end{align*}
We have:
$$I\left(W_1,j,\{1,2\}\right)=-I(W_1,j',\{1,2\}),$$
and:
$$I_\iota\left(W_1,j,\{1,2\}\right)=-I_\iota(W_1,j',\{1,2\}),$$
with Notation \ref{not_intiota}.
\end{lem}
\begin{proof}
 Define the following diffeomorphism:
$$
f \colon
\left\{
\begin{array}{lcl}
F(\{1,2\}) &\rightarrow &F(\{1,2\}) \\
(c_0 \colon \{1,3,4\} \rightarrow \check{M}, c_1 \colon \{1,2\} \rightarrow T_{c_0(1)}M)
&\mapsto&
(c_0',c_1),
\end{array}
\right.
$$
where:
$$
c_0' :
\left\{
\begin{array}{lcl}
 3 &\mapsto& c_0(4) \\
 4 &\mapsto& c_0(3) \\
 1 &\mapsto& c_0(1),
\end{array}
\right..
$$
Then $f$ is orientation-reversing, and furthermore:
\begin{align*}
 p_{e_{34}} \circ f &= p_{e_{43}}, \\
 p_{e_{43}} \circ f &= p_{e_{34}}, \\
 p_{e_{13}} \circ f &= p_{e_{24}},\\
 p_{e_{24}} \circ f &= p_{e_{13}}.
\end{align*}
Note that these equalities correspond to the relation between the labellings $j$ and $j'$. That being checked, and using the change of variables $\indi\times f$, we find:
$$I(W_1,j',\{1,2\})=-I(W_1,j,\{1,2\}).$$
\end{proof}
Using Lemma \ref{lem_reverse}, we now get similar equalities for $W_2$.
\begin{lem}\label{lem_doubleedge1_d}
Let $j$ be in $\bijlab(E(W_2))$. Let $j'\in\bijlab(E(W_2))$ be the bijection only differing from $j$ by $j'(e_{34})=j(e_{43})$ and $j'(e_{43})=j(e_{34})$.
We have:
$$I(W_2,j,\{3,4\})=-I(W_2,j',\{3,4\}).$$
\end{lem}
\begin{proof}
 Use the second equality of Lemma \ref{lem_doubleedge1} and Lemma \ref{lem_reverse}.
\end{proof}
\begin{lem}\label{lem_doubleedge2_d}
Let $j$ be in $\bijlab(E(W_2))$. Let $j'\in\bijlab(E(W_2))$ be the bijection differing from $j$ by:
\begin{align*}
 j'(e_{24})&=j(e_{13}), \\
 j'(e_{13})&=j(e_{24}), \\
 j'(e_{43})&=j(e_{34}), \\
 j'(e_{34})&=j(e_{43}).
 \end{align*}
We have:
$$I\left(W_2,j,\{1,2\}\right)=-I(W_2,j',\{1,2\}).$$
\end{lem}
\begin{proof}
    Use the second equality of Lemma \ref{lem_doubleedge2} and Lemma \ref{lem_reverse}.
\end{proof}

\subsection{Triangular faces}

Let $\Gamma$ be in $\{T_1,T_2\}$. Let $j$ be in $\bijlab(E(\Gamma))$. Let $A \subset \{1,2,3,4\}$ be a set of cardinality $3$. The pair $(\Gamma,F(A))$ is called a \emph{triangular face}. The graph given by the edges in $\Gamma$ between the vertices in $A$ is a triangle whose edges are oriented. We distinguish two cases, see Figure \ref{triangle_face}. First, assume that this triangle is an oriented cycle. Consider the edge of $E_A(\Gamma)$ with the greatest label among the three edges in $E_A(\Gamma)$, and let $v$ be the vertex of $A$ outside this edge. In the second case, let $v$ be the only vertex in $A$ that is the first vertex of an edge in $E_A(\Gamma)$ and the second vertex of another edge in $E_A(\Gamma)$. In both cases, let $w$ be the vertex of $A$ such that there is an edge going from $v$ to $w$, and let $x$ be the remaining vertex of $A$. 

\begin{figure}
\centering
\begin{subfigure}{0.4\textwidth}
\centering
\begin{tikzpicture}[scale=1.7]
    \useasboundingbox (-0.7,-0.5) rectangle (0.7,0.7);
    \coordinate (v) at (90:0.5);
    \coordinate (a) at (-30:0.5);
    \coordinate (b) at (-150:0.5);
    \draw[very thick]  (v)--(a) node[midway,sloped]{$>$};
    \draw (30:0.48) node{\scriptsize $e_{vw}$};
    \draw[very thick]  (a)--(b) node[midway,sloped]{$<$};
    \draw[very thick]  (b)--(v) node[midway,sloped]{$>$};
    \draw (150:0.48) node{\scriptsize $e_{xv}$};
    \draw (v) node[above]{\scriptsize $v$};
    \draw (a) node[below]{\scriptsize $w$};
    \draw (b) node[below]{\scriptsize $x$};
    \draw (-90:0.4) node{\scriptsize $e_{wx}$};
\end{tikzpicture}
    \caption{First case: $j(e_{vw}),j(e_{xv})<j(e_{wx})$.}
\end{subfigure}
\begin{subfigure}{0.4\textwidth}
\centering
\begin{tikzpicture}[scale=1.7]
    \useasboundingbox (-0.7,-0.5) rectangle (0.7,0.7);
    \coordinate (v) at (90:0.5);
    \coordinate (a) at (-30:0.5);
    \coordinate (b) at (-150:0.5);
    \draw[very thick] (v)--(a) node[midway,sloped]{$>$};
     \draw (30:0.48) node{\scriptsize $e_{vw}$};
    \draw[very thick]  (a)--(b) node[midway,sloped]{$>$};
    \draw[very thick]  (b)--(v) node[midway,sloped]{$>$};
    \draw (150:0.48) node{\scriptsize $e_{xv}$};
    \draw (v) node[above]{\scriptsize$v$};
    \draw (a) node[below]{\scriptsize$w$};
    \draw (b) node[below]{\scriptsize$x$};
\end{tikzpicture}
    \caption{Second case.}
\end{subfigure}

 \caption{Notation for a triangular face.}  
 \label{triangle_face}
\end{figure}
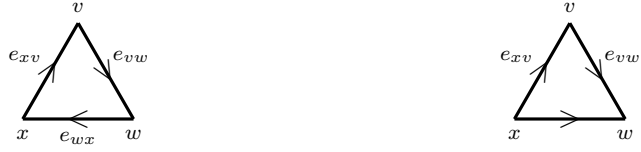

\begin{lem}\label{lem_triangle}
Let $j'\in \bijlab(E(\Gamma))$ be the bijection obtained from $j$ by exchanging $j(e_{vw})$ and $j(e_{xv})$.  We have:
$$I(\Gamma,j,A)=-I(\Gamma,j',A).$$
\end{lem}
\begin{proof}
We define a diffeomorphism from $F(A)$ to $F(A)$ as follows:
$$
f\colon
\left\{
\begin{array}{lcl}
 F(A) &\rightarrow& F(A) \\
 (c_0 \colon (\{1,2,3,4\}\setminus A)\cup\{v\} \rightarrow \check{M}, c_1 \colon A \rightarrow T_{c_0(v)}M)&\mapsto& (c_0,c_1'),
\end{array}
\right.
$$
where:
$$
c_1' \colon
\left\{
\begin{array}{lcl}
 v &\mapsto& c_1(w)+c_1(x)-c_1(v) \\
 w &\mapsto& c_1(w) \\
 x &\mapsto& c_1(x).
\end{array}
\right.
$$
Since $c_1$ is an injection, the map $c_1'$ is also an injection. Thus, the map $f$ is well-defined. The class of $c_1'$ in $\check{S}_A(T_{c_0(v)}M)$ does not depend on the particular choice of $c_1$ as a representative of its class in $\check{S}_A(T_{c_0(v)}M)$. Moreover $f$ is orientation-reversing. Finally, we have $p_{e_{vw}}\circ f=p_{e_{xv}}$, $p_{e_{xv}}\circ f=p_{e_{vw}}$, and $p_e \circ f=p_e$ for $e\in E(\Gamma)\setminus \{e_{vw},e_{xv}\}$. Indeed, keeping our notation, we have:
$$p_{e_{vw}}\circ f(c)=\frac{c_1'(w)-c_1'(v)}{||c_1'(w)-c_1'(v)||}=\frac{c_1(v)-c_1(x)}{||c_1(v)-c_1(x)||}=p_{e_{xv}}(c).$$
We conclude, using $\indi\times f$ as a change of variables.
\end{proof}

\subsection{Proof of Proposition \ref{prop_main} up to one-edge faces}
Let $\Gamma$ be in $\{T_1,T_2,W_1,W_2\}$. Let $A$ be a subset of $\{1,2,3,4\}$ such that $|E_A(\Gamma)|=1$. The pair $(\Gamma,F(A))$ is called a \emph{one-edge face}. The third section is devoted to the proof of the following proposition.
\begin{prop}\label{prop_IHX}
We have:
$$\sum_{j\in\bijlab(E(T_1))} \sum_{\substack{A\subset \{1,2,3,4\} \\ |E_A(T_1)|=1}} I(T_1,j,A)+ 3 \sum_{j\in\bijlab(E(W_1))} \sum_{\substack{A\subset \{1,2,3,4\} \\ |E_A(W_1)|=1}} I(W_1,j,A) =0,$$
and the same equality obtained by replacing $I$ with $I_\iota$.
\end{prop}
We now prove Proposition \ref{prop_main} up to Proposition \ref{prop_IHX}. Lemmas \ref{lem_inf} and \ref{lem_inf_bis} show that the contribution of faces where a bunch of points goes to infinity vanishes. Combining results about other degenerate faces, anomalous faces, double-edge faces and triangular faces, we get the following proposition. It shows that the contribution from faces that are neither infinite faces nor one-edge faces vanishes.
\begin{prop}\label{prop_autresfaces}
Let $\Gamma$ be in $\{T_1,T_2,W_1,W_2\}$. We have:
$$\sum_{j\in\bijlab(E(\Gamma))} \sum_{\substack{A\subset \{1,2,3,4\} \\ |E_A(\Gamma)|\neq 1}} I(\Gamma,j,A)=0.$$
 \end{prop}
 \begin{proof}
    See Lemmas \ref{lem_degenerateW}, \ref{lem_degsansedge}, \ref{lem_anomalous}, \ref{lem_doubleedge1}, \ref{lem_doubleedge2}, \ref{lem_doubleedge1_d}, \ref{lem_doubleedge2_d} and \ref{lem_triangle}. 
 \end{proof}
Using Lemma \ref{lem_reverse}, we deduce the following proposition from Proposition \ref{prop_IHX}. Together, these two propositions show that the contribution of one-edge faces vanishes.
 \begin{prop} \label{prop_IHX_d}
We have:
$$\sum_{j\in\bijlab(E(T_2))} \sum_{\substack{A\subset \{1,2,3,4\} \\ |E_A(T_2)|=1}} I(T_2,j,A)+ 3 \sum_{j\in\bijlab(E(W_2))} \sum_{\substack{A\subset \{1,2,3,4\} \\ |E_A(W_2)|=1}} I(W_2,j,A) =0.$$
\end{prop}
\begin{proof}
This equality follows from Lemma \ref{lem_reverse} and Proposition \ref{prop_IHX}. 
\end{proof}
Therefore this concludes the proof of Proposition \ref{prop_main}.

\section{One-edge faces}

\label{sec_one}
We will rephrase Proposition \ref{prop_IHX} in Section \ref{sec_rephrase}. See Proposition \ref{prop_rephrase}. For this we give some new definitions in Sections \ref{triv_orient} and \ref{sec_weight}. Some of these definitions can be found in Chapters $6.3$ and $7.1$ of \cite{lesbookzv2}.

\subsection{Trivalent graphs and orientations of configuration spaces} \label{triv_orient}
In the previous sections we worked with the concrete graphs $T_1,W_1,T_2,W_2$. For our purposes, it is now convenient to introduce abstract trivalent graphs. 

\begin{defi} \label{trivalent}
A \emph{trivalent graph} $\Gamma$ is a set $H(\Gamma)$ equipped with two partitions, $E(\Gamma)$ and $V(\Gamma)$. The elements of $H(\Gamma)$ are called \emph{half-edges}. The elements of $E(\Gamma)$ are unordered pairs of half-edges. They are called \emph{edges}. The elements of $V(\Gamma)$ are unordered triples of half-edges. They are called \emph{vertices}. An \emph{isomorphism of trivalent graphs} between $\Gamma$ and $\Gamma'$ is a bijection $H(\Gamma) \rightarrow H(\Gamma')$ compatible with the edge partition and the vertex partition. An \emph{automorphism of a trivalent graph} $\Gamma$ is an isomorphism from $\Gamma$ to itself.
\end{defi}
 Consider connected trivalent graphs with four vertices without looped edges. Such graphs have six edges. Up to isomorphism, there are two of them: \raisebox{-1.2mm}\tetra and \raisebox{-1.2mm}{\tatata}. We call \emph{tetrahedron graph} the first one, and \emph{double-theta} the second. The group of automorphisms of a tetrahedron graph is identified with the group of permutations of its vertices. The group of automorphisms of a double-theta is generated by the two obvious planar reflections in the drawing and the automorphism that permutes the two edges in the left-hand double-edge and preserves each vertex. The double-theta has $16$ automorphisms.

Let $\Gamma$ be a tetrahedron graph or a double-theta. We define an orientation of $C_{V(\Gamma)}(M)$ associated with an orientation of the set of vertices of $\Gamma$.\footnote{An \emph{orientation} of a finite set $X$ is a total order on $X$ up to permutation of signature $+1$.}

\begin{defi}
Let $\Gamma$ be a tetrahedron graph or a double-theta.  Let $b \colon V(\Gamma) \rightarrow \{1,2,3,4\}$ be a representative of an orientation of $V(\Gamma)$ and let $f_b \colon C_{V(\Gamma)}(M) \rightarrow C_4(M)$ be the associated diffeomorphism. We define the \emph{orientation of $C_{V(\Gamma)}(M)$ associated with the orientation of $V(\Gamma)$} represented by $b$ to be the orientation of $C_{V(\Gamma)}(M)$ which makes $f_b$ orientation-preserving.
\end{defi}

We now define an orientation of $C_{V(\Gamma)}(M)$ associated with a pair (edge-orientation,vertex-orientation) of $\Gamma$. See Definition \ref{def_vo}.
 \begin{defi}
A \emph{vertex-orientation} for a trivalent graph $\Gamma$ is the data of a cyclic order of the three half-edges belonging to the same vertex, for each vertex. An \emph{edge-orientation} is an order of the two half-edges belonging to the same edge, for each edge.
\end{defi}
\begin{defi}
Let $\Gamma$ be a trivalent graph immersed\footnote{An \emph{immersion of $\Gamma$ in $\R^2$} is the data of an injection $i \colon  V(\Gamma) \rightarrow \R^2$ and of a collection of proper embeddings $i_e \colon [0,1]\rightarrow\R^2$ for every edge $e=\{v_1,v_2\}\in E(\Gamma)$ such that:
\begin{itemize}
    \item $i_e(0)=i(v_1)$ and $i_e(1)=i(v_2)$,
    \item if $e_1,e_2,e_3$ are three edges adjacent to a vertex $v$, there is a small neighbourhood $D_v$ of $i(v)$ and an orientation-preserving homeomorphism $h_v \colon D_v \rightarrow D^2$ which sends $v$ to $0$ and the $i_{e_j}([0,1])\cap D_v$ to $[0,1],[0,\exp(\frac{2i\pi}{3})],[0,\exp(\frac{4i\pi}{3})]$. 
\end{itemize}} in $\R^2$. We define the \emph{vertex-orientation of $\Gamma$ associated with the immersion} to be the following: take the counterclockwise cyclic order of the half-edges at each vertex. In the following, immersed trivalent graphs will always be equipped with the vertex-orientation associated with the immersion.
\end{defi}

Let $b$ be an orientation of $V(\Gamma)$. We first define an orientation of $H(\Gamma)$ from $b$ and a vertex-orientation of $\Gamma$: it is the concatenation of the cyclic orientations of the half-edges adjacent to the same vertex given by the vertex-orientation, using the orientation $b$ of $V(\Gamma)$. We now define another orientation on $H(\Gamma)$ associated with an edge-orientation of $\Gamma$: it is the concatenation of the orientations of the pair of half-edges belonging to the same edge, using any orientation for the set of edges. These two orientations of $H(\Gamma)$ are used in the following definition. See also \cite[Chapter $7$, Section $1$]{lesbookzv2}.
\begin{defi}\label{def_vo}
Let $\Gamma$ be equipped with a vertex-orientation and edge-orientation. Let $b$ be the orientation of $V(\Gamma)$ such that the orientation of $H(\Gamma)$ associated with $b$ and the vertex-orientation of $\Gamma$ coincides with the orientation of $H(\Gamma)$ associated with the edge-orientation of $\Gamma$. The \emph{orientation of $C_{V(\Gamma)}(M)$ associated with the pair (edge-orientation, vertex-orientation) of $\Gamma$} is the orientation associated with $b$.
\end{defi}

In the following lemma, we compute these two orientations in the cases of $T_1$ and $W_1$ and show that they coincide. See Remark \ref{rque_match_or}.

\begin{lem}\label{lem_match_or}
    Let $\Gamma$ be in $\{T_1,W_1\}$. The orientation of $C_{V(\Gamma)}(M)$ associated with the pair (edge-orientation, vertex-orientation) coincides with the orientation of $C_{V(\Gamma)}(M)$ associated with the numbering of the vertices.
\end{lem}
\begin{proof}
Let $\Gamma$ be $T_1$. In Figure \ref{T_1new}, the graph $T_1$ is drawn and its half-edges are labelled. With this notation, we have:
$$H(T_1)=\{a_1,a_2,b_1,b_2,c_1,c_2,d_1,d_2,e_1,e_2,f_1,f_2\},$$
and:
$$V(T_1)=\{\{c_1,b_1,f_1\},\{a_2,b_2,e_1\},\{e_2,c_2,d_1\},\{d_2,f_2,a_1\}\}.$$
The vertex-orientation $o(T_1)$ of $T_1$ associated with the immersion in Figure \ref{T_1new} is represented by:
$$\{(c_1,b_1,f_1),(a_2,b_2,e_1),(e_2,c_2,d_1),(d_2,f_2,a_1)\}.$$
The orientation of $H(T_1)$ associated with the numbering of $V(T_1)$ and the vertex-orientation $o(T_1)$ is represented by:
$$(c_1,b_1,f_1,a_2,b_2,e_1,e_2,c_2,d_1,d_2,f_2,a_1).$$
The orientation of $H(T_1)$ associated with the edge-orientation of the edges of $T_1$ is represented by:
$$(a_1,a_2,b_1,b_2,c_1,c_2,d_1,d_2,e_1,e_2,f_1,f_2).$$
The sequence \eqref{eq_check_or_T1} of orders shows that these two orientations coincide. This implies that the orientation of $C_{V(T_1)}(M)$ associated with the numbering of $V(T_1)$ coincides with the orientation of $C_{V(T_1)}(M)$ associated with the pair (edge-orientation, vertex-orientation). We start from the first order and apply $3$-cycles to the underlined triples. We freely move overlined pairs $(x_1,x_2)$.
\begin{equation}\label{eq_check_or_T1}
\begin{aligned}
(c_1,b_1,\underline {f_1,a_2,b_2},\overline{e_1,e_2},c_2,\overline{d_1,d_2},f_2,a_1) \\
(d_1,d_2,e_1,e_2,c_1,\overline{b_1,b_2},f_1,\underline{a_2,c_2,f_2},a_1) \\
(b_1,b_2,d_1,d_2,e_1,e_2,c_1,\overline{f_1,f_2},\underline{a_2,c_2,a_1}) \\
(b_1,b_2,d_1,d_2,e_1,e_2,f_1,f_2,\overline{c_1,c_2},\overline{a_1,a_2}).
\end{aligned}
\end{equation}
Let $\Gamma$ be $W_1$. With the notation of Figure \ref{TT_1new}, we have:
$$V(W_1)=\{\{f_1,b_1,c_1\},\{e_1,c_2,b_2\},\{d_2,f_2,a_1\},\{a_2,e_2,d_1\}\}.$$ The vertex-orientation $o(W_1)$ of $W_1$ associated with the drawing in Figure \ref{TT_1new} is represented by:
 $$o(W_1)=\{(f_1,b_1,c_1),(e_1,c_2,b_2),(d_2,f_2,a_1),(a_2,e_2,d_1)\}.$$
The orientation of $H(W_1)$ obtained as the ordered concatenation of the cyclic orderings of the half-edges at each vertex is represented by:
$$(f_1,b_1,c_1,e_1,c_2,b_2,d_2,f_2,a_1,a_2,e_2,d_1).$$ 
The sequence \eqref{eq_check_or_W1} of orders shows that this orientation coincides with the orientation of $H(W_1)$ associated with the edge-orientation of $W_1$. We use the same strategy as before.
\begin{equation}\label{eq_check_or_W1}
\begin{aligned}
(f_1,b_1,c_1,\underline{e_1,c_2,b_2},d_2,f_2,\overline{a_1,a_2},e_2,d_1) \\
(a_1,a_2,f_1,b_1,\overline{c_1,c_2},b_2,e_1,\underline{d_2,f_2,e_2},d_1) \\
(a_1,a_2,c_1,c_2,f_1,\overline{b_1,b_2},\overline{e_1,e_2},\underline{d_2,f_2,d_1}) \\
(a_1,a_2,b_1,b_2,c_1,c_2,e_1,e_2,f_1,f_2,\overline{d_1,d_2}).
 \end{aligned}
 \end{equation}
\end{proof}

\begin{rque}\label{rque_match_or}
    A permutation of $\{1,2,3,4\}$ of signature $+1$ induces an orientation-preserving diffeomorphism of $C_4(M)$. In the proofs of Theorem \ref{thm_secondcount} and Theorem \ref{thm_Lcount}, a coherent choice of the numberings (up to even permutation) of the four graphs in Figure \ref{graphs} is needed. Our choice was precisely made in order to obtain Lemma \ref{lem_match_or}. 
\end{rque}

Let $e$ be in $E(\Gamma)$. Let $p_e \colon C_{V(\Gamma)}(M) \rightarrow C_2(M)$ be the composition of the natural restriction map from $C_{V(\Gamma)}(M)$ to $C_e(M)$ by the map from $C_e(M)$ to $C_2(M)$ induced by the order of the two vertices in $e$ given by the edge-orientation. 

Let $b \colon V(\Gamma) \rightarrow \{1,2,3,4\}$ be any bijection. Let $A$ be a subset of $V(\Gamma)$ such that $|A|\geq1$. We define:
$$F(\Gamma,A)=f_b^{-1}\bigl(F(b(A))\bigr),$$
and:
$$I(\Gamma,j,o(\Gamma),A)=\int_{[0,1]\times F(\Gamma,A)} \bigwedge_{e\in E(\Gamma)} (\indi\times p_e)^*(\omega_{j(e)}),$$
where $C_{V(\Gamma)}(M)$ is equipped with the orientation associated with the edge-orientation of $\Gamma$ and $o(\Gamma)$, and where $F(\Gamma,A)$ is oriented as part of the boundary of $C_{V(\Gamma)}(M)$.

\subsection{A space of diagrams and a weight system}\label{sec_weight}
Consider the function $w$ defined on $\{T_1,T_2,W_1,W_2\}$ by $w(T_1)=w(T_2)=1$ and $w(W_1)=w(W_2)=2$. It is the restriction of a linear function $w$ on a particular vector space $A^c_2(\emptyset)$, see Lemma \ref{def_lin}. This space $A^c_2(\emptyset)$ is a quotient of the $\R$-vector space $D_2^c(\emptyset)$ formally generated by vertex-orientation-preserving isomorphism classes of connected trivalent graphs with four vertices equipped with a vertex-orientation.\footnote{Here, \say{c} stands for \say{connected} and $2$ is half the number of vertices. The $\emptyset$ symbol is there to keep our notation consistent with \cite{lesbookzv2}, where more general spaces of diagrams are defined.} See Definition \ref{def_jac}. Such functions are often called \emph{weight systems}. We use $w$ to rephrase Proposition \ref{prop_IHX}. See Proposition \ref{prop_rephrase}.

\begin{defi}\label{def_jac}
An \emph{antisymmetry relation} is an element of $D_2^c(\emptyset)$ of the form $[\Gamma] + [\Gamma']$ where $\Gamma$ and $\Gamma'$ are immersed in $\R^2$ with identical images outside a disk neighbourhood of a vertex where they look like Figure \ref{AS}.

A \emph{Jacobi relation} is an element of $D_2^c(\emptyset)$ of the form $[\Gamma_1] + [\Gamma_2]+ [\Gamma_3]$ where $\Gamma_1$, $\Gamma_2$ and $\Gamma_3$ are immersed in $\R^2$ with identical images outside of a disk neighbourhood of a vertex where they look like Figure \ref{jaco}.

We let $\A_2^c(\emptyset)$ be the quotient of $D_2^c(\emptyset)$ by antisymmetry and Jacobi relations.
 \end{defi}
 
\begin{figure}
\centering
\begin{subfigure}{0.4\textwidth}
\centering
\antisym
\caption{Antisymmetry relation.}
\label{AS}
\end{subfigure}
\begin{subfigure}{0.4\textwidth}
\centering
\Jacobi
\caption{Jacobi relation.}
\label{jaco}
\end{subfigure}
\end{figure}

\begin{figure}
\centering
\begin{tikzpicture}
    \useasboundingbox (-1,-1.25) rectangle (7,0.5);

    \coordinate (a) at (-1,0);
    \coordinate (b) at (1,0);
    \coordinate (c) at (0,-1);
    \coordinate (d) at (-0.5,-0.5);
    \draw[very thick] (a)--(c);
\draw[very thick] (b)--(c);
\draw[very thick] (a) .. controls (0,0.25) .. (b);
\draw[very thick] (a) .. controls (0,-0.25) .. (b);
\draw[very thick] (c) .. controls (-0.5,-1).. (d);

\draw (a) node{$\bullet$};
\draw (b) node{$\bullet$};
\draw (c) node{$\bullet$};
\draw (d) node{$\bullet$};

\draw[very thick] (1.25,-0.5)--(1.75,-0.5);
\draw[very thick] (1.5,-0.25)--(1.5,-0.75);

\begin{scope}[xshift=3cm]
\coordinate (e) at (-1,0);
    \coordinate (f) at (1,0);
    \coordinate (g) at (0,-1);
    \coordinate (h) at (0,-0.25);
    \draw[very thick] (e)--(g);
\draw[very thick] (f)--(g);
\draw[very thick] (e) .. controls (0,0.25) .. (f);
\draw[very thick] (e)--(h)--(f);
\draw[very thick] (g) .. controls (-0.5,-1).. (h);

\draw (e) node{$\bullet$};
\draw (f) node{$\bullet$};
\draw (g) node{$\bullet$};
\draw (h) node{$\bullet$};

\draw[very thick] (1.25,-0.5)--(1.75,-0.5);
\draw[very thick] (1.5,-0.25)--(1.5,-0.75);
    
\end{scope}

\begin{scope}[xshift=6cm]
\coordinate (i) at (-1,0);
    \coordinate (j) at (1,0);
    \coordinate (k) at (0,-1);
    \coordinate (l) at (0,0.25);
    \draw[very thick] (i)--(k);
\draw[very thick] (j)--(k);
\draw[very thick] (i)--(l)--(j);
\draw[very thick] (i) .. controls (0,-0.25) .. (j);
\draw[very thick] (k) .. controls (-0.5,-1).. (l);

\draw (i) node{$\bullet$};
\draw (j) node{$\bullet$};
\draw (k) node{$\bullet$};
\draw (l) node{$\bullet$};
    
\end{scope}
    
\end{tikzpicture}
\caption{A Jacobi relation.}
\label{jac_relation}
\end{figure}
The following lemma is classical. It will be used as a definition.
\begin{lem} \label{def_lin}
There is a linear form $w$ on $\A_2^c(\emptyset)$ such that $w\left(\left[\raisebox{-1.3mm}{\tetra}\right]\right)=1$ and $w\left(\left[\raisebox{-1.5mm}{\tatata}\right]\right)=2.$ 
\end{lem}
\begin{proof}
Note that if $[\Gamma,o(\Gamma)]$ is the class of a graph with a loop edge equipped with a vertex-orientation, then $2[\Gamma,o(\Gamma)]$ is an antisymmetry relation. Let $\Gamma$ be a tetrahedron graph or a double-theta. Any automorphism of $\Gamma$ acts on the set of vertex-orientations of $\Gamma$ by changing the vertex-orientation at an even number of vertices. Hence one can define a linear form $w$ with the required properties on the quotient of $D_c^2(\emptyset)$ by antisymmetry relations. 

 The group of automorphisms of a tetrahedron graph acts transitively on its edges. Thus, up to antisymmetry, the only Jacobi relation involving a tetrahedron graph is the Jacobi relation of Figure \ref{jac_relation}. It is sent to $0$ by $w$. Note that this Jacobi relation involves one double-theta. Jacobi relations which do not involve a tetrahedron graph are in the linear span of the antisymmetry relations. Thus $w$ induces a linear form on $\A_2^c(\emptyset)$ with the required properties.
\end{proof}
\begin{defi}
     A \emph{labelling} of a graph is an element $j$ of $\bijlab(E(\Gamma))$. A \emph{labelled edge-oriented trivalent graph} $\Gamma$ is a trivalent graph equipped with an edge-orientation and a labelling. An \emph{isomorphism of labelled edge-oriented trivalent graphs}  between $(\Gamma,j)$ and $(\Gamma',j')$ is an isomorphism of trivalent graphs $f \colon H(\Gamma) \rightarrow H(\Gamma')$ compatible with the edge orientations, and such that $j'\circ f =j$, where $f$ also denotes the induced map from $E(\Gamma)$ to $E(\Gamma')$. 
\end{defi}

\begin{exemple}\label{ex_auto_graphs}
The graph $T_1$ has $3$ automorphisms compatible with its edge-orientation.  The graph $W_1$ has $2$ automorphisms compatible with its edge-orientation. 
\end{exemple}

Let $\Gamma$ be in $\{T_1,W_1\}$. We define a \emph{labelled $\Gamma$} to be a pair $(\Gamma',j')$ where $j'$ is a labelling of $\Gamma'$ and $\Gamma'$ is an edge-oriented trivalent graph isomorphic to $\Gamma$ as an edge-oriented graph. Let $D_1$ be the set of isomorphism classes of labelled $T_1$ and $W_1$. As a consequence of our definitions, we have the following lemma.

\begin{lem}\label{lem_rephrase_u}
We have the equality:
\begin{multline*}
\sum_{[(\Gamma,j)]\in D_1} \sum_{\substack{A\subset V(\Gamma)\\|E_A(\Gamma)|=1}} w(\Gamma,o(\Gamma))I(\Gamma,j,o(\Gamma),A)= \frac{1}{3} \sum_{j\in\bijlab(E(T_1))} \sum_{\substack{A\subset \{1,2,3,4\}\\|E_A(T_1)|=1}} I(T_1,j,A)  \\
+\sum_{j\in\bijlab(E(W_1))} \sum_{\substack{A\subset \{1,2,3,4\}\\|E_A(W_1)|=1}} I(W_1,j,A).
\end{multline*}
\end{lem}
\begin{proof}
We first check that the left-hand side of the above equality is well-defined. Let $(\Gamma,j)$ be a labelled $T_1$ (resp. $W_1$). Let $f$ be an isomorphism of edge-oriented graphs from $H(T_1)$ (resp. $H(W_1)$) to $H(\Gamma)$. Let $A$ be a subset of $V(\Gamma)$ such that $|A|\geq 2$. Note that $w(\Gamma,o(\Gamma))I(\Gamma,j,o(\Gamma),A)$ does not depend on the chosen vertex-orientation $o(\Gamma)$. Choosing $o(\Gamma)$ to be $f(o(T_1))$ (resp. $f(o(W_1))$) shows that it is equal to
$I(T_1,j\circ f,f^{-1}(A))$
(resp. to
$2I(W_1,j\circ f,f^{-1}(A))$). Let $j_1$ be any labelling of $T_1$ (resp. of $W_1$) such that $(\Gamma,j)$ and $(T_1,j_1)$ (resp. $(W_1,j_1)$) are isomorphic as labelled edge-oriented graphs.
The quantity
$$\sum_{\substack{A\subset V(\Gamma)\\|E_A(\Gamma)|=1}} 
w(\Gamma,o(\Gamma))I(\Gamma,j,o(\Gamma),A)$$
only depends on the isomorphism class of $(\Gamma,j)$ as a labelled edge-oriented graph. Indeed, it is equal to $$\sum_{\substack{A\subset \{1,2,3,4\}\\|E_A(\Gamma)|=1}} 
I(T_1,j_1,A),$$
resp. to
$$2\sum_{\substack{A\subset \{1,2,3,4\}\\|E_A(\Gamma)|=1}} 
I(W_1,j_1,A).$$

Now consider the projection from the set of pairs $(T_1,j_1)$ (resp. $(W_1,j_1)$) with $j_1$ a labelling of $T_1$ (resp. of $W_1$) on $D_1$ that sends $(\Gamma,j_1)$ to $[(\Gamma,j_1)]$. This projection is onto, and each isomorphism class of labelled $T_1$ (resp. of labelled $W_1$) has three (resp. two) preimages under this projection. See Example \ref{ex_auto_graphs}. This concludes the proof of the lemma.
\end{proof}
If $\Gamma$ is a labelled $T_1$ or $W_1$, the product $w(\Gamma,o(\Gamma))I(\Gamma,j,o(\Gamma),A)$ does not depend on the chosen vertex-orientation $o(\Gamma)$. We will denote it by $w(\Gamma)I(\Gamma,j,A)$. As a corollary to Lemma \ref{lem_rephrase_u}, to prove Proposition \ref{prop_IHX} it suffices to prove the following equivalent proposition. 
\begin{prop}\label{prop_rephrase}
We have:
$$
\sum_{[(\Gamma,j)]\in D_1} \sum_{\substack{A\subset V(\Gamma)\\|E_A(\Gamma)|=1}} w(\Gamma)I(\Gamma,j,A)=0.
$$

\end{prop}
\subsection{Sketch of proof of Proposition \ref{prop_rephrase}}
\label{sec_rephrase}

Let us prove Proposition \ref{prop_rephrase}. For $l\in\pert$, let $D_1(l) \subset D_1$ be the set of $[(\Gamma,j)]\in D_1$ such that $|E_{j^{-1}(l)}(\Gamma)|=1$.\footnote{When $e$ is in $e_{uv}(\Gamma)$, we let $E_e(\Gamma)$ be $E_{\{u,v\}}(\Gamma)$.} We write:
$$\sum_{[(\Gamma,j)]\in D_1} \sum_{\substack{A\subset V(\Gamma)\\|E_A(\Gamma)|=1}} w(\Gamma)I(\Gamma,j,A)= \sum_{l\in\pert} \sum_{[(\Gamma,j)]\in D_1(l)} w(\Gamma)I(\Gamma,j,j^{-1}(l)).$$
Let $l\in\pert$ be a label. We now prove the following equality:
 $$\sum_{[(\Gamma,j)]\in D_1(l)} w(\Gamma)I(\Gamma,j,j^{-1}(l))=0.$$
 For a representative $(\Gamma,j)$ of an element in $D_1(l)$, define $c_l(\Gamma,j)$ to be the class of the labelled graph obtained from $\Gamma$ by collapsing the edge $j^{-1}(l)$. We obtain a function $c_l$ from $D_1(l)$ to the set of isomorphism classes of edge-oriented graphs with $3$ vertices of valences $3,3,4$ labelled by $\pert \setminus \{l\}$. The preimages of $c_l$ are a partition of $D_1(l)$.

Let us analyse the image of $c_l$. Let $(\Gamma,j)$ be a representative of an element of $D_1(l)$. The graph $\Gamma$ has a unique vertex $v_1$ that is the source of three edges. Consider the edges $e$ of $\Gamma$ such that $|E_{e}(\Gamma)|=1$. Among those, we call \emph{type-two edges} the edges adjacent to $v_1$, and \emph{type-one edges} the others. See Figure \ref{edges}. Let $D_1(l,1)$ (resp. $D_1(l,2)$) be the set of $[(\Gamma,j)]\in D_1(l)$ such that $j^{-1}(l)$ is a type-one (resp. a type-two) edge. The set $D_1(l)$ is the disjoint union of $D_1(l,1)$ and $D_1(l,2)$.

Consider the function $g$ defined on the image of $c_l$ that forgets the labelling. Then $|g\circ c_l(D_1(l,i))|=1$ for $i\in \{1,2\}$. See Figure \ref{graph_collapse}. Moreover, the two graphs depicted in Figure \ref{graph_collapse} are not isomorphic to each other as edge-oriented graphs. Indeed, they don't have the same number of edges that leave the unique vertex of valence four. The image of $c_l$ is the set of isomorphism classes of pairs $(\Gamma,j)$ where $\Gamma$ is one of the two graphs of Figure \ref{graph_collapse} and $j\colon E(\Gamma)\rightarrow \pert \setminus \{l\}$ is a bijection. Each set of our partition of $D_1(l)$ contains at least one labelled $T_1$.
 
 In Section \ref{sec_typeone} we describe $c_l^{-1}(c_l(\{[\Gamma_1,j_1]\}))$ where $(\Gamma_1,j_1)$ is a labelled $T_1$ in $D_1(l,1)$. See Lemma \ref{lem_preimage1}. In Lemma \ref{lem_IHXu} we show that:
$$\sum_{[(\Gamma,j)]\in c_l^{-1}(c_l(\{[\Gamma_1,j_1]\}))} w(\Gamma)I(\Gamma,j,j^{-1}(l))=0.$$
In Section \ref{sec_typetwo}, we describe $c_l^{-1}(c_l(\{[\Gamma_1',j_1']\}))$ where $(\Gamma_1',j_1')$ is a labelled $T_1$ in $D_1(l,2)$. See Lemma \ref{lem_preimage2}. In Lemma \ref{lem_IHXd} we show that:
$$\sum_{[(\Gamma,j)]\in c_l^{-1}(c_l(\{[\Gamma_1',j_1']\}))} w(\Gamma)I(\Gamma,j,j^{-1}(l))=0.$$
This concludes the proof of Propositon \ref{prop_rephrase} up to Lemmas \ref{lem_IHXu} and \ref{lem_IHXd}.
 
\begin{figure}
\centering
\begin{subfigure}{0.4\textwidth}
\centering
\begin{tikzpicture}
\useasboundingbox (-1,-1) rectangle (1,1);
\coordinate (a) at (0,1);
\coordinate (b) at (0,0);
\coordinate (c) at (0,-1);

\draw (a)--(b) node[midway, sloped]{$>$};
\draw (a).. controls (1,-0.25) .. (c) node[midway, sloped]{$>$} ;
\draw (a).. controls (-1,-0.25) .. (c) node[midway, sloped]{$>$} ;
\draw (b).. controls (0.5,-0.5) .. (c) node[midway, sloped]{$<$} ;
\draw (b).. controls (-0.5,-0.5) .. (c) node[midway, sloped]{$>$};
\draw (a) node{$\bullet$};
\draw (b) node{$\bullet$};
\draw (c) node{$\bullet$};
\end{tikzpicture}

\caption{On $D_1(l,1)$.}
\end{subfigure}
\begin{subfigure}{0.4\textwidth}
\centering
\begin{tikzpicture}
\useasboundingbox (-1,-1) rectangle (1,1);
\coordinate (a) at (0,1);
\coordinate (b) at (-30:1);
\coordinate (c) at (-150:1);

\draw (a)--(b) node[midway, sloped]{$<$};
\draw (b)--(c) node[midway, sloped]{$>$};
\draw (c)--(a) node[midway, sloped]{$<$};
\draw (a).. controls (180:1) .. (c) node[midway, sloped]{$<$} ;
\draw (a).. controls (0:1) .. (b) node[midway, sloped]{$>$} ;
\draw (a) node{$\bullet$};
\draw (b) node{$\bullet$};
\draw (c) node{$\bullet$};
\end{tikzpicture}
\caption{On $D_1(l,2)$.}
\end{subfigure}

\caption{The map $g\circ c_l$.}
\label{graph_collapse}
\end{figure}
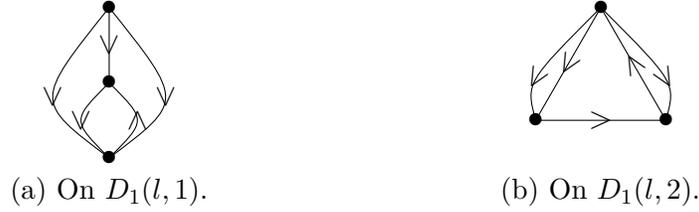

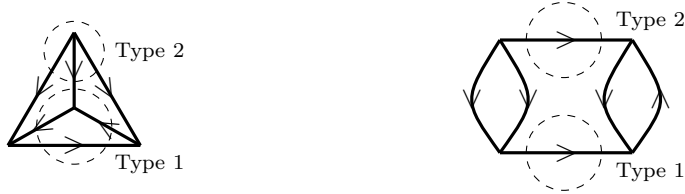
\begin{figure}[H]
\centering
\begin{subfigure}{0.4\textwidth}
\centering
\typeoneandtwo
\end{subfigure}
\begin{subfigure}{0.4\textwidth}
\centering
  \begin{tikzpicture}
\useasboundingbox (4.75,-0.7) rectangle (7,1.2);
\begin{scope}[yshift=0.3 cm]
\draw[very thick] (5,-0.5) -- (6.76,-0.5) node[midway,sloped]{$>$};
\draw[very thick] (5,1) -- (6.76,1) node[midway,sloped]{$>$};
\draw[very thick] (5,-0.5) .. controls (5.5,0.2) ..(5,1) node[midway,sloped]{$<$};
\draw[very thick] (5,-0.5) .. controls (4.5,0.2) ..(5,1) node[midway,sloped]{$<$};
\draw[very thick] (6.76,-0.5) .. controls (7.26,0.2) ..(6.76,1) node[midway,sloped]{$>$};
\draw[very thick] (6.76,-0.5) .. controls (6.26,0.2) ..(6.76,1) node[midway,sloped]{$<$};
\draw[dashed] (5.85,-0.5) circle (0.5);
\draw[dashed] (5.85,1) circle (0.5);
\draw (7,1.25) node{\scriptsize Type $2$};
\draw (7,-0.75) node{\scriptsize Type $1$};
\end{scope}
\end{tikzpicture}  
\end{subfigure}
\caption{Type $1$ and type $2$ edges.}
\label{edges}
\end{figure}

\subsection{Type-one edges}\label{sec_typeone}

Let $(\Gamma_1,j_1)$ be a labelled $T_1$ in $D_1(l,1)$. Let $e=j^{-1}_1(l)$. Let $v_2$ and $v_3$ be its two vertices, such that $e$ goes from $v_2$ to $v_3$. Let $v_1$ be the vertex adjacent to the three type-two edges, and $v_4$ the remaining vertex. Let $a=e_{v_4v_2}$, $b=e_{v_1v_2}$, $c=e_{v_1v_3}$, $d=e_{v_3v_4}$ and $f=e_{v_1v_4}$. For each edge $x$, let $x_1$ and $x_2$ denote its two ordered half-edges. See Figure \ref{T_1new}. We have the following fact.

\begin{figure}
    \centering
    \begin{subfigure}{0.3\textwidth}
     \resizebox{\textwidth}{!}{
\Tunnew}
    \caption{The graph $(\Gamma_1,j_1)(c)$.}
    \label{T_1new}
    \end{subfigure}
    \begin{subfigure}{0.3\textwidth}
    \centering
     \resizebox{\textwidth}{!}{
    \Wunnew}
    \caption{The graph $(\Gamma_1,j_1)(a)$.}
    \label{TT_1new}
\end{subfigure}
\begin{subfigure}{0.3\textwidth}
    \centering
     \resizebox{\textwidth}{!}{
    \TunIHXnew}
    \caption{The graph $(\Gamma_1,j_1)(b)$.}
    \label{T_1IHXnew}
\end{subfigure}
\begin{subfigure}{0.3\textwidth}
\resizebox{\textwidth}{!}{
\begin{tikzpicture}[scale=1.5]
\useasboundingbox (-1,-0.7) rectangle (1,1.2);
\draw[very thick] (-0.88,-0.5) -- (0.88,-0.5) node[midway,sloped]{$>$} node[near start, below,scale=0.7]{\scriptsize$e_1$} node[near end, below,scale=0.7]{\scriptsize$e_2$};
\draw[very thick] (0.88,-0.5) -- (0,0) node[midway,sloped]{$<$} node[near end, below,scale=0.7]{\scriptsize$d_2$};
\draw (0.55,-0.4) node[scale=0.7]{\scriptsize$d_1$};
\draw[very thick] (0,0) -- (-0.22,-0.125) node[near end, sloped]{$<$} node[near start, below,scale=0.7]{\scriptsize$a_1$} ;
\draw[very thick] (-0.66,-0.375) -- (-0.88,-0.5) node[near start, sloped]{$<$};
\draw[very thick] (0,1) -- (-0.22,0.675) node[near end,sloped]{$<$} ;
\draw[very thick] (-0.66,-0.175) -- (-0.88,-0.5) node[near start,sloped]{$<$} ;
\draw (-0.3,0.75) node[scale=0.7]{\scriptsize$b_1$};
\draw[very thick] (0,1) -- (0,0) node[midway,sloped]{$>$};
\draw (0.08,0.75) node[scale=0.7]{\scriptsize$f_1$};
\draw (0.08,0.25) node[scale=0.7]{\scriptsize$f_2$};
\draw[very thick] (0,1) --(0.22,0.625) node[near end,sloped]{$>$} ;
\draw[very thick] (0.66,-0.175) --(0.88,-0.5) node[near start,sloped]{$>$}  ;
\draw (0.3,0.75) node[scale=0.7]{\scriptsize$c_1$};
\draw (0.7,0) node[scale=0.7]{\scriptsize$x_2$};
\draw (-0.88,-0.5) node{$\bullet$};
\draw (0.88,-0.5) node{$\bullet$};
\draw (0,0) node{$\bullet$};
\draw (0,1) node{$\bullet$};

\end{tikzpicture}}
    \caption{Constructing $(\Gamma_1,j_1)(x)$.}
    \label{construction}
\end{subfigure}
\begin{subfigure}{0.4\textwidth}
\centering
\begin{tikzpicture}[scale=1.8]
\useasboundingbox (-1,-1.2) rectangle (1,1);
\coordinate (a) at (0,1);
\coordinate (b) at (0,0);
\coordinate (c) at (0,-1);

\draw[very thick] (a)--(b) node[midway, sloped]{$>$};
\draw[very thick] (a).. controls (1,-0.25) .. (c) node[midway, sloped]{$>$} ;
\draw[very thick] (a).. controls (-1,-0.25) .. (c) node[midway, sloped]{$>$} ;
\draw[very thick] (b).. controls (0.5,-0.5) .. (c) node[midway, sloped]{$<$} ;
\draw[very thick] (b).. controls (-0.5,-0.5) .. (c) node[midway, sloped]{$>$};
\draw (a) node{$\bullet$};
\draw (b) node{$\bullet$};
\draw (c) node{$\bullet$};
\draw (a)++(-0.3,0) node{\scriptsize$b_1$};
\draw (a)++(0.3,0) node{\scriptsize$c_1$};
\draw (a)++(0.15,-0.3) node{\scriptsize$f_1$};
\draw (a)++(0.15,-0.75) node{\scriptsize$f_2$};
\draw (b)++(0.3,0) node{\scriptsize$d_2$};
\draw (b)++(-0.3,0) node{\scriptsize$a_1$};
\draw (c)++(-0.3,0) node{\scriptsize$b_2$};
\draw (c)++(0.3,0) node{\scriptsize$c_2$};
\draw (c)++(-0.2,0.4) node{\scriptsize$a_2$};
\draw (c)++(0.2,0.4) node{\scriptsize$d_1$};
\end{tikzpicture}
\caption{The collapse of $(\Gamma_1,j_1)$ along $e$.}
\end{subfigure}
\caption{The set $c_l^{-1}(c_l(\{[\Gamma_1,j_1]\}))$.}
\label{fig_set_typeone}
\end{figure}

\begin{lem}\label{lem_preimage1}
The set $c_l^{-1}(c_l(\{[\Gamma_1,j_1]\}))$ has three elements. An explicit set of representatives is given by the three labelled edge-oriented graphs of Figures \ref{T_1new}, \ref{TT_1new}, and \ref{T_1IHXnew}.
\end{lem}

\begin{proof}
Let $(\Gamma,j)$ be an edge-oriented labelled trivalent graph such that the graph obtained by collapsing $\Gamma$ along $j^{-1}(l)$ is a representative of $c_l([\Gamma_1,j_1])$. Up to isomorphism of labelled edge-oriented graphs, the graph $\Gamma$ satisfies:
\begin{enumerate}
  \item  $H(\Gamma)=H(\Gamma_1)$, 
  \item $E(\Gamma)=E(\Gamma_1)$,
  \item the edge-orientation of $\Gamma$ is the same as the edge-orientation of $\Gamma_1$,
  \item $j=j_1$,
  \item $\{c_1,b_1,f_1\}$ and $\{a_1,d_2,f_2\}$ are two vertices of $\Gamma$. 
\end{enumerate}
Assume moreover that $(\Gamma,j)$ is a labelled $T_1$ or $W_1$. There is a unique vertex of $\Gamma$ from at least two edges start. It is $\{c_1,b_1,f_1\}$. So $d_1$ and $e_1$ belong to different vertices. The graph $(\Gamma,j)$ is thus determined by the half-edge belonging to the same vertex as $e_2$ and $d_1$. If this half-edge is $a_2$ (resp. $b_2$,$c_2$), we let $(\Gamma_1,j_1)(a)$ (resp. $(\Gamma_1,j_1)(b)$, $(\Gamma_1,j_1)(c)$) be the graph obtained in this way. The graph $(\Gamma_1,j_1)(a)$ is a labelled $W_1$. The graph $(\Gamma_1,j_1)(b)$ is a labelled $T_1$, and we have:
$$[(\Gamma_1,j_1)(b)]\neq [(\Gamma_1,j_1)(c)].$$
\end{proof}

\begin{lem}\label{lem_IHXu}
We have:
$$\sum_{[(\Gamma,j)]\in c_l^{-1}(c_l(\{[\Gamma_1,j_1]\}))} w(\Gamma)I(\Gamma,j,j^{-1}(l))=0.$$
\end{lem}
\begin{proof}
Applying Lemma \ref{lem_preimage1}, it suffices to prove the following equality:
   \begin{align*}
w(\Gamma_1(c))I((\Gamma_1,j_1)(c),e)+w(\Gamma_1(b))I((\Gamma_1,j_1)(b),e) +w(\Gamma_1(a))I((\Gamma_1,j_1)(a),e)=0.
\end{align*}
Let $x$ be in $\{a,b,c\}$. We equip $V(\Gamma_1(x))$ with the numbering given in Figure \ref{fig_set_typeone}. We let $o(\Gamma_1(x))$ be the vertex-orientation of $\Gamma_1(x)$ associated with the drawing in Figure \ref{fig_set_typeone}. Lemma \ref{lem_match_or} implies that the orientation of $C_{V(\Gamma_1(x))}(M)$ associated with the pair (edge-orientation, $o(\Gamma_1(x))$) is the same as the orientation associated with the orientation of $V(\Gamma_1(x))$ represented by $i\mapsto v_i$. Lemma \ref{def_lin} implies
$$w(\Gamma_1(c),o(\Gamma_1(c)))-w(\Gamma_1(a),o(\Gamma_1(a)))+w(\Gamma_1(b),o(\Gamma_1(b)))=0.$$
Let us prove that:
$$I((\Gamma_1,j_1)(c),o(\Gamma_1(c)),e)=I((\Gamma_1,j_1)(b),o(\Gamma_1(b)),e)=-I((\Gamma_1,j_1)(a),o(\Gamma_1(a)),e).$$
Consider the bijection between $V(\Gamma_1(c))$ and $V(\Gamma_1(a))$ induced by the transposition $(34)$. This map induces a diffeomorphism $f$ from $F(\Gamma_1(c),e)$ to $F(\Gamma_1(a),e)$. The diffeomorphism $f$ is orientation-reversing. Finally we have on $[0,1]\times F(\Gamma_1(c),e)$:
$$(\indi\times f)^* \left(\bigwedge_{x\in E(\Gamma_1(a))} (\indi\times p_x)^*(\omega_{j_1(a)(x)})\right)=\bigwedge_{x\in E(\Gamma_1(c))} (\indi\times p_x)^*(\omega_{j_1(a)(c)}).$$
This proves:
$$I((\Gamma_1,j_1)(c),o(\Gamma_1(c)),e)=-I((\Gamma_1,j_1)(a),o(\Gamma_1(a)),e).$$
The equality $I((\Gamma_1,j_1)(c),o(\Gamma_1(c)),e)=I((\Gamma_1,j_1)(b),o(\Gamma_1(b)),e)$ follows from similar considerations. Indeed the bijection between $V(\Gamma_1(c))$ and $V(\Gamma_1(b))$ associated with $\id_{\{1,2,3,4\}}$ induces a diffeomorphism between $F(\Gamma_1(c),e)$ and $F(\Gamma_1(b),e)$. This diffeomorphism is orientation-preserving. We conclude as in the previous paragraph.
\end{proof}

\subsection{Type-two edges} \label{sec_typetwo}
Let $(\Gamma_1',j_1')$ be a labelled $T_1$ in $D_1(l,2)$. Let $e=(j_1')^{-1}(l)$. Let $v_1$ and $v_2$ be its two vertices, such that $e$ goes from $v_1$ to $v_2$. Let $v_3$ be the vertex connected to $v_2$ by an edge going from $v_2$ to $v_3$, and $v_4$ the remaining vertex. Let $a=e_{v_1v_4}$, $b=e_{v_1v_3}$, $c=e_{v_2v_3}$, $d=e_{v_4v_2}$, and $f=e_{v_3v_4}$. See Figure \ref{T_1bisnew}.

\begin{figure}
    \centering
    \begin{subfigure}{0.3\textwidth}
    \resizebox{\textwidth}{!}{
    \Tunbisnew}
    \caption{The graph $(\Gamma_1',j_1')(c)$.}
    \label{T_1bisnew}
    \end{subfigure}
    \begin{subfigure}{0.3\textwidth}
    \resizebox{\textwidth}{!}{
    \Wunbisnew}
    \caption{The graph $(\Gamma_1',j_1')(a)$.}
    \label{TT_1bisnew}
\end{subfigure}
\begin{subfigure}{0.3\textwidth}
\resizebox{\textwidth}{!}{
    \TunIHXbisnew}
    \caption{The graph $(\Gamma_1',j_1')(b)$.}
    \label{T_1IHXbisnew}
\end{subfigure}
\begin{subfigure}{0.3\textwidth}
\resizebox{\textwidth}{!}{
\begin{tikzpicture}[scale=1.5]
\useasboundingbox (-1,-0.7) rectangle (1,1.2);
\draw[very thick] (-0.88,-0.5) -- (-0.44,-0.5) node[near end, sloped]{$>$}; 
\draw[very thick] (0.44,-0.5)--(0.88,-0.5) node[near start, sloped]{$>$} node[near end, below,scale=0.7]{\scriptsize$c_2$};

\draw[very thick] (0.88,-0.5) -- (0,0) node[midway,sloped]{$<$} node[near end, below,scale=0.7]{\scriptsize$f_2$};
\draw (0.55,-0.4) node[scale=0.7]{\scriptsize$f_1$};

\draw[very thick] (0,0) -- (-0.88,-0.5) node[midway, sloped]{$<$} node[near start, above,scale=0.7]{\scriptsize$d_1$} ;
\draw (-0.6,-0.2) node[scale=0.7]{\scriptsize$d_2$};

\draw[very thick] (0,1) -- (-0.88,-0.5) node[midway,sloped]{$<$} ;
\draw (-0.3,0.75) node[scale=0.7]{\scriptsize$e_1$};
\draw (-0.7,0) node[scale=0.7]{\scriptsize$e_2$};

\draw[very thick] (0,1) -- (0,0.75) node[near end, sloped]{$>$};
\draw[very thick] (0,0.25) -- (0,0) node[near start, sloped]{$>$};
\draw (0,0.35) node[scale=0.7]{\scriptsize$a_2$};

\draw[very thick] (0,1) -- (0.22,0.625) node[near end,sloped]{$>$}  ;
\draw[very thick] (0.66,-0.125)--(0.88,-0.5) node[near start,sloped]{$>$}  ;
\draw (0.7,0) node[scale=0.7]{\scriptsize$b_2$};
\draw (-0.88,-0.5) node{$\bullet$};
\draw (0.88,-0.5) node{$\bullet$};
\draw (0,0) node{$\bullet$};
\draw (0,1) node{$\bullet$};

\end{tikzpicture}}
\caption{Constructing $(\Gamma_1',j_1')(x)$}
\label{construction_deux}
\end{subfigure}
\begin{subfigure}{0.4\textwidth}
\centering
\begin{tikzpicture}[scale=2.1]
\useasboundingbox (-1,-0.7) rectangle (1,1);
\coordinate (a) at (0,1);
\coordinate (b) at (-30:1);
\coordinate (c) at (-150:1);

\draw[very thick] (a)--(b) node[midway, sloped]{$<$};
\draw[very thick] (b)--(c) node[midway, sloped]{$>$};
\draw[very thick] (c)--(a) node[midway, sloped]{$<$};
\draw[very thick] (a).. controls (180:1) .. (c) node[midway, sloped]{$<$} ;
\draw[very thick] (a).. controls (0:1) .. (b) node[midway, sloped]{$>$} ;
\draw (a) node{$\bullet$};
\draw (b) node{$\bullet$};
\draw (c) node{$\bullet$};
\draw (a)++(-0.3,0) node{\scriptsize$c_1$};
\draw (a)++(0.3,0) node{\scriptsize$a_1$};
\draw (a)++(-0.1,-0.4) node{\scriptsize$b_1$};
\draw (a)++(0.1,-0.4) node{\scriptsize$d_2$};
\draw (c)++(0.2,-0.1) node{\scriptsize$f_1$};
\draw (b)++(-0.2,-0.1) node{\scriptsize$f_2$};
\draw (c)++(-0.3,0.1) node{\scriptsize$c_2$};
\draw (c)++(0.4,0.3) node{\scriptsize$b_2$};
\draw (b)++(0.3,0.1) node{\scriptsize$a_2$};
\draw (b)++(-0.4,0.3) node{\scriptsize$d_1$};
\end{tikzpicture}
\caption{The collapse of $(\Gamma_1',j_1')$ along $e$.}
\end{subfigure}
\caption{The set $c_l^{-1}(c_l(\{[\Gamma_1',j_1']\}))$.}
\label{fig_set_typetwo}
\end{figure}
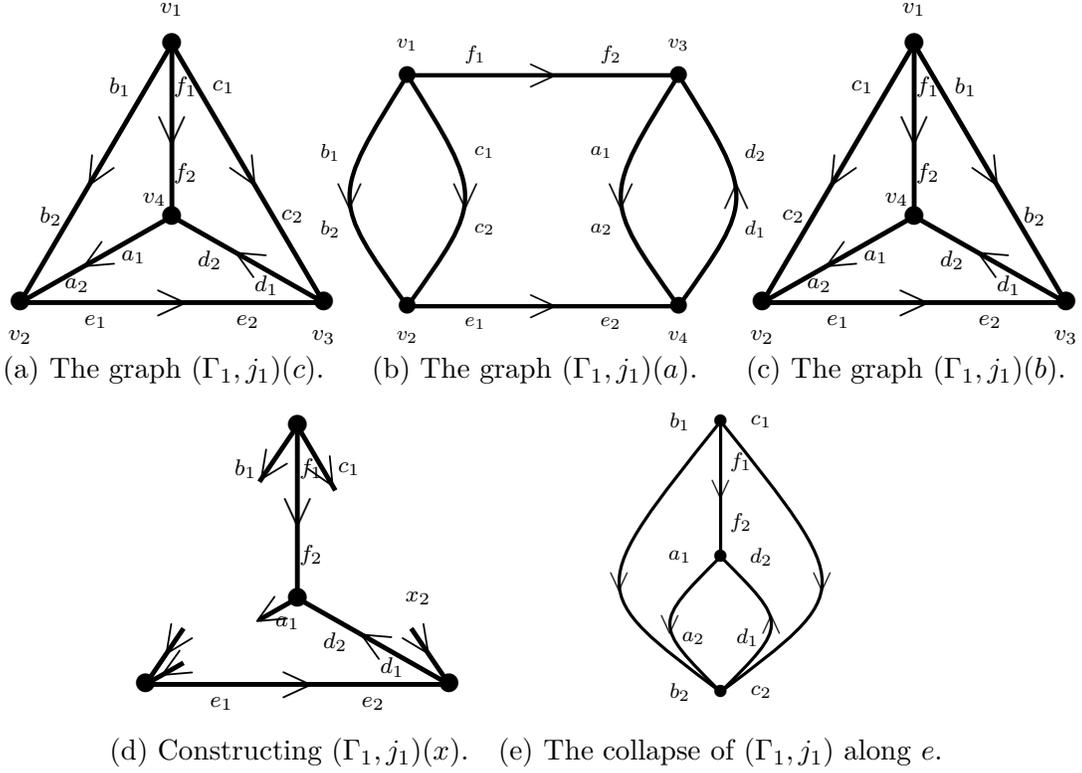

\begin{lem}\label{lem_preimage2}
The set $c_l^{-1}(c_l(\{[\Gamma_1',j_1']\}))$ has three elements. An explicit set of representatives is given by the three labelled edge-oriented graphs of Figures \ref{T_1bisnew}, \ref{TT_1bisnew}, and \ref{T_1IHXbisnew}.
\end{lem}

\begin{proof}
Let $(\Gamma',j')$ be an edge-oriented labelled trivalent graph such that the graph obtained by collapsing $\Gamma'$ along $(j')^{-1}(l)$ is a representative of $c_l([\Gamma_1',j_1'])$. Up to isomorphism of labelled edge-oriented graphs, the graph $\Gamma'$ satisfies:
\begin{enumerate}
  \item  $H(\Gamma')=H(\Gamma_1')$, 
  \item $E(\Gamma')=E(\Gamma_1')$,
  \item the edge-orientation of $\Gamma$ is the same as the edge-orientation of $\Gamma_1'$,
  \item $j'=j_1'$,
  \item  $\{f_2,a_2,d_1\}$ and $\{c_2,f_1,b_2\}$ are vertices of $\Gamma'$.
\end{enumerate}
Assume moreover that $(\Gamma',j')$ is a labelled $T_1$ or $W_1$. Then one vertex of $\Gamma'$ is the starting point of three edges. So $d_2$ and $e_2$ belong to the same vertex. The graph $(\Gamma',j')$ is thus determined by the remaining half-edge in the vertex $\{e_2,d_2,*\}$. If this half-edge is $a_1$ (resp. $b_1$ or $c_1$), we let $(\Gamma_1',j_1')(a)$ (resp. $(\Gamma_1',j_1')(b)$, $(\Gamma_1',j_1')(c)$) be the labelled graph obtained in this way. The graph $(\Gamma_1',j_1')(a)$ is a labelled $W_1$. The graph $(\Gamma_1',j_1')(b)$ is a labelled $T_1$, and we have:
$$[(\Gamma_1',j_1')(c)]\neq [(\Gamma_1',j_1')(b)].$$
\end{proof}

\begin{lem}\label{lem_IHXd}
We have:
$$\sum_{[(\Gamma,j)]\in c_l^{-1}(c_l(\{[\Gamma_1',j_1']\}))} w(\Gamma)I(\Gamma,j,j^{-1}(l))=0.$$
\end{lem}
\begin{proof}
   The proof is similar to the proof of Lemma \ref{lem_IHXu}. Applying Lemma \ref{lem_preimage2}, it suffices to prove the following equality:
   \begin{align*}
w(\Gamma_1'(c))I((\Gamma_1',j_1')(c),e)+w(\Gamma_1'(b))I((\Gamma_1',j_1')(b),e)+w(\Gamma_1'(a)))I((\Gamma_1',j_1')(a),e)=0.
\end{align*}
Let $x$ be in $\{a,b,c\}$. We equip $V(\Gamma_1'(x))$ with the numbering given in Figure \ref{fig_set_typetwo}. We let $o(\Gamma_1'(x))$ be the vertex-orientation associated with the drawing in Figure \ref{fig_set_typetwo}. Lemma \ref{lem_match_or} implies that the orientation of $C_{V(\Gamma_1'(x))}(M)$ associated with the pair (edge-orientation, $o(\Gamma_1'(x))$) coincides with the orientation of $C_{V(\Gamma_1'(x))}(M)$ associated with the orientation of $V(\Gamma_1'(x))$.
The bijection from $V(\Gamma_1'(c))$ to $V(\Gamma_1'(a))$ associated with the transposition $(23)$ induces an orientation-reversing diffeomorphism from  $F(\Gamma_1'(c),e)$ to $F(\Gamma_1'(a),e)$. This proves:
$$I((\Gamma_1',j_1')(c),o(\Gamma_1'(c)),e)=-I((\Gamma_1',j_1')(a),o(\Gamma_1'(a)),e).$$
   We conclude as in the proof of Lemma \ref{lem_IHXu}.
\end{proof}

\section{Towards applications}
\label{sec_app}

In this section, we give more illustrations of the interest of our Theorem \ref{thm_secondcount}. Recall from the end of Section \ref{sec_main} that we are computing $\lambda_2$ from a Heegaard diagram of the manifold using Morse propagating chains and forms in a work in progress. In Section \ref{sec_transversality}, we provide additional details explaining why our definition of $(z_{KKT})_2$ is more suitable for concrete computations than the original one. In particular, we show why the needed transversality in discrete computations requires the use of distinct propagating chains. In Section \ref{sec_LMO_KKT}, we show that knowing the values of $\lambda_2$ on the lens spaces $L(p,1)$ for all prime numbers $p$ is enough to determine the difference between the invariants $(\zlmo)_2$ and $(\zkkt)_2$.

\subsection{Why transversality requires distinct propagating chains}

\label{sec_transversality}

Recall that our main Theorem \ref{thm_secondcount} and Lescop's definition of $(\zkkt)_2$ (Theorem \ref{thm_Lcount}) both have a \emph{dual version} where propapating forms and integrals are replaced by propagating chains and algebraic intersections. In order to apply the dual version of Theorem \ref{thm_Lcount} (resp. of Theorem \ref{thm_secondcount}) for a family of six propagating chains $(P_i)_{i\in\pert}$, one must prove that for every graph $\Gamma\in \T\cup\W$ (resp. $\Gamma\in\{T_1,T_2,W_1,W_2\}$) and every $j\in\bijlab(E(\Gamma))$ the intersection $\bigcap_{e\in E(\Gamma)} p_e^{-1}(P_{j(e)})$ is transverse. If that is the case, we say that the family $(P_i)_{i\in\pert}$ satisfies the transversality condition for Theorem \ref{thm_Lcount} (resp. of Theorem \ref{thm_secondcount}) for $M$ (the ambient manifold). In this section, we check whether or not these conditions are satisfied in some examples. In particular, we show the following facts.
\begin{itemize}\item  There exists a propagating chain $P$ of $S^3$ such that the constant family $(P_i=P)_{i\in\pert}$ satisfies the transversality condition for Theorem \ref{thm_secondcount}. See Example \ref{ex_non_transverse}.
 \item In contrast, for any rational homology sphere $M$ and any propagating chain $P$ of $M$, the constant family $(P_i=P)_{i\in\pert}$ does not satisfy the transversality condition for Theorem \ref{thm_Lcount}. See Examples \ref{ex_non_transverse} and \ref{ex_non_transverse_inf}.
\end{itemize}

The second fact justifies the last claim of Remark \ref{rque_simplier}: for any rational homology sphere, we must to be able to use distinct propagating chains to achieve transversality in the dual version of Theorem \ref{thm_Lcount}. The first fact shows that this claim is not true for the dual version of Theorem \ref{thm_secondcount} in the case of $S^3$. In Remark \ref{rque_trans_lambda}, we explain why it is likely to be true for the dual version of Theorem \ref{thm_secondcount} in the case of a general rational homology sphere.

\begin{exemple}\label{ex_non_transverse}
  Let $a$ be in $S^2$. Recall that $G_{S^3}^{-1}(a)$ is a propagating chain of $S^3$. Below, we show that there exists $T \in \T$ such that the intersection of the $p_e^{-1}(G_{S^3}^{-1}(a))$ over $E(T)$ is not transverse. Consider the graph $T_<$ obtained from $T_1$ by reversing the orientation of the edge $e_{24}$. Let $(t_1,t_2,t_3)\in\R^3$ be such that $0<t_1<t_2<t_3$. The configuration $(0,t_1a,t_2a,t_3a)$ is in the intersection of the preimages of $G_{S^3}^{-1}(a)$ under the maps associated to the edges of $T_<$. We have a three-parameter family in the intersection. So the intersection is not transverse.

  In contrast, for any graph $\Gamma$ in our family $\{T_1,T_2,W_1,W_2\}$ the intersection of the $p_e^{-1}(G_{S^3}^{-1}(a))$ over $E(\Gamma)$ is empty. This gives an easy direct computation of $\tilde{\lambda}_2(S^3)=\lambda_2(S^3)=0$ with a single propagating chain. (This result can also be deduced from Theorem \ref{thm_Lcount} using three distinct propagating chains. Indeed, let $b,c\in S^2$ be such that $a,b,c$ are not coplanar. Let $(P_i)_{i\in\pert}$ be a family of six propagating chains  consisting of four propagators $G_{S^3}^{-1}(a)$, one propagator $G_{S^3}^{-1}(b)$ and one propagator $G_{S^3}^{-1}(c)$. Then for any graph $\Gamma\in\T\cup\W$ the intersection of the $p_e^{-1}(P_{j(e)})$ over $E(\Gamma)$ is empty.)
  
\end{exemple}

Example \ref{ex_non_transverse_inf} is an infinitesimal version of Example  \ref{ex_non_transverse}. It shows that the dual version of Theorem \ref{thm_Lcount} could never be applied
with the same propagating chain on every edge, because non-transverse intersections would always occur on the boundary of $C_4(M)$. We first need the following remark.

\begin{rque}
Recall that $U\check{M}$ is part of $\partial C_2(M)$. Below we show that any propagating chain $P$ of $M$ must intersect any fiber of $U\check{M}$. The second homology group of $C_2(M)$ with rational coefficients $H_2(C_2(M);\Q)$ is generated by the homology class $[S]$ of a fiber $S$ of the unit tangent bundle to $\check{M}$. Any propagating chain $P$ of $M$ satisfies the following homological property.
For any $2$-cycle $F$ of $C_2(M)$ transverse to $P$,
we have $[F]=\langle P,F\rangle_{C_2(M)} [S] $ in $H_2(C_2(M);\Q)=\Q [S]$, where $\langle P,F\rangle_{C_2(M)}$ denotes the algebraic intersection of $P$ and $F$ in $C_2(M)$.
So $P$ must intersect any fiber of $U\check{M}$.
\end{rque}

\begin{exemple}\label{ex_non_transverse_inf}
Let $\gamma \colon ]-1,1[ \to \check{M}$ be an immersion such that the direction $a$ of $\gamma^{\prime}(0)$ at $m=\gamma(0)$ is in $P$. Let $(t_1,t_2) \in \R^2$ be such that
$1<t_1<t_2$. Then the limit at $0$ of \begin{equation*}
 \Bigl(\bigl(\gamma(0), \gamma(t), \gamma(tt_1),\gamma(tt_2) \bigr) \in \check{C}_4(\check{M}) \Bigr)_{t \in ]0,1[}                                     
\end{equation*}
 exists in $C_4(M)$. It belongs to $\partial C_4(M)$. Denote it by $c_{m,a}(t_1,t_2)$. Recall the graph $T_<$ from Example \ref{ex_non_transverse}. 
The restriction map associated to each edge of $T_<$ maps the degenerate configuration $c_{m,a}(t_1,t_2)$ to $a$. So associating the same propagating chain to every edge of $T_<$ yields a degenerate configuration $c_{m,a}(t_1,t_2)$ in the intersection of the preimages of $P$ under the maps associated to the edges of $T_<$. Therefore, the intersection could never be transverse with Lescop’s definition of $(\zkkt)_2$ with the same propagating chain at every edge.
(The configurations $c_{m,a}(t_1,t_2)$ are actually all distinct, and we even have a $2$-parameter family in the intersection.) This is why we have to use several propagating chains in general to get transverse intersections.
\end{exemple}

\begin{rque}\label{rque_trans_lambda}
Let $P$ be a propagating chain of $C_2(M)$. Assume that $P$ intersects $U\check{M}$ as a section of the unit tangent bundle. (This can be achieved.) Associate $P$ to every edge of our graphs $T_1$, $T_2$, $W_1$ and $W_2$. The restrictions of the edge-orientations and the existence of oriented cycles \cycletheta and \cycletetra in our formula guarantee that no degenerate configuration as above is in the intersection of the preimages of $P$ under the maps associated to the edges for any oriented graph in our collection.

The preimages of $P$ under the two maps associated to the parallel edges of a double-theta $W_1$ or $W_2$ coincide in $C_4(M)$. So any non-empty intersection associated with a double-theta $W_1$ or $W_2$ where the two parallel edges are equipped with $P$ is not transverse. This is why we also use distinct propagating chains in our case. 
\end{rque}

\subsection{About the identification between the degree two parts of the invariants LMO and KKT}
\label{sec_LMO_KKT}

In this section, we compare the degree two parts of the invariants $\zlmo$ and $\zkkt$. We will need the following well-known property of $\lambda_2=\tilde{\lambda}_2$.

\begin{lem}
 Let $M$ be a rational homology sphere. We have 
 $$\lambda_2(-M)=\lambda_2(M).$$
\end{lem}
\begin{proof}
  Note that $C_2(M)=C_2(-M)$. Recall from the discussion after Proposition \ref{bordcdeu} that in order to define the Gauss map $G_{-M}$ we need an orientation-preversing diffeomorphism $\phi_{-M,\infty}$ between a neighbourhood of $\infty\in (-M)$ and $\mathring{B}_{1,\infty}$. We can choose $\phi_{-M,\infty}$ to be $-\phi_{M,\infty}$. We get $G_{-M}=-G_M$. Thus, if $\omega$ is a propagating form on $M$, then $-\omega$ is a propagating form on $(-M)$. To apply Theorem \ref{thm_secondcount} to $(-M)$, we integrate a wedge-product of an even number of propagating forms of $(-M)$ over $C_4(-M)=C_4(M)$. The result follows. 
\end{proof}

Let $M$ be a rational homology sphere.  For a prime number $p$, let $\nu_p(M)$ denote the $p$-adic valuation of $|H_1(M;\Z)|$. We have
$$|H_1(M;\Z)|=\prod_{p \mbox{ \scriptsize prime }} p^{\nu_p(M)}.$$
Note that $\nu_p(M)=\nu_p(-M)$ for any prime number $p$. 
The invariant $(\zlmo)_2$ is valued in the same space as $(\zkkt)_2$. We define a numerical invariant $\lambda_{2,\mbox{\scriptsize LMO }}$ of rational homology spheres by the formula $$(\zlmo)_2 = \lambda_{2,\mbox{\scriptsize LMO }} \left[\tetraeq\right].$$
We also have $\lambda_{2,\mbox{\scriptsize LMO }}(-M)=\lambda_{2,\mbox{\scriptsize LMO }}(M)$.
In the following proposition, we show that a theorem of Moussard (\cite[Proposition 1.11]{moussardAGT}) implies that $\lambda_2$ and $\lambda_{2,\mbox{\scriptsize LMO }}$  coincide up to a combination of the invariants $\nu_p$.

\begin{prop}\label{prop_LMO_KKT}
We have
$$\lambda_2-\lambda_{2,\mbox{\scriptsize LMO }}= \sum_{p \mbox{ \scriptsize prime }} \Bigl(\lambda_2(L(p,1)) - \lambda_{2,\mbox{\scriptsize LMO }}(L(p,1))\Bigr)\nu_p.$$
Note that the right-hand side sum is finite when evaluated at a $\Q$-sphere.
 \end{prop}

\begin{proof}
Say that an invariant is \emph{additive} if it is addtive with respect to the connected sum. As in Moussard's article, let $I_{n}^c$ denote the vector space of real-valued additive finite type invariants of degree at most $n$ of rational homology spheres. With Moussard's terminology, the invariants $\nu_p$ for $p$ a prime number are of degree $1$ (Proposition 1.9). By Corollary 1.10 of \cite{moussardAGT}, the space $I_1^c$ is generated by the maps $\nu_p$ for prime numbers $p$.

 The Casson-Walker invariant $\lambda_{\mbox{\scriptsize CW}}$ is an additive invariant of degree $2$. Proposition 1.11 of \cite{moussardAGT} determines the quotients $I_{n}^c/I_{n-1}^c$ for any integer $n>1$.  It implies that the quotient $I_2^c/I_1^c$ is one-dimensional, generated by the class of $\lambda_{\mbox{\scriptsize CW}}$, and that $I_3^c=I_2^c$. 
 
 Universality properties with respect to finite type invariants of rational homology spheres were proven for $\zkkt$ by Lescop in \cite{lessumgen} (see also \cite[Theorem 18.5]{lesbookzv2}), and for $\zlmo$ by Massuyeau in \cite{massuyeausplit}. These properties imply in particular that the invariants $\lambda_2$ and $\lambda_{2,\mbox{\scriptsize LMO }}$ are additive finite type invariants of degree $4$, and that the classes of $\lambda_2$ and $\lambda_{2,\mbox{\scriptsize LMO }}$ in $I_4^c/I_3^c$ coincide. \footnote{More generally, the universality properties imply that $(\zkkt)_n$ and $(\zlmo)_n$ are additive finite type invariants of degree $2n$ and coincide on rational homology spheres up to additive invariants of degree $2n-1$ for any integer $n\geq 1$.}

Thus there exist real numbers $\alpha$, and $\alpha_p$ for each prime number $p$, such that
$$\lambda_2-\lambda_{2,\mbox{\scriptsize LMO }}=\alpha \lambda_{\mbox{\scriptsize CW}} + \sum_{p \mbox{ \scriptsize prime }} \alpha_p \nu_p.$$
For any rational homology sphere $M$, we have $\lambda_{\mbox{\scriptsize CW}}(-M)=-\lambda_{\mbox{\scriptsize CW}}(M)$. Let $M$ be a rational homology sphere $M$ such that $\lambda_{\mbox{\scriptsize CW}}(M)\neq 0$. Applying the above formula to $M$ and $-M$, we find $\alpha=0$. Applying the above formula to $L(p,1)$ gives $\alpha_p = \lambda_2(L(p,1))-\lambda_{2,\mbox{\scriptsize LMO }}(L(p,1))$. This concludes the proof.

\end{proof}

In particular, computing $\lambda_2$ for the lens spaces by applying
Theorem \ref{thm_secondcount} with Morse propagators would allow us to characterize $\lambda_2$.

\def\cprime{$'$}
\providecommand{\bysame}{\leavevmode ---\ }
\providecommand{\og}{``}
\providecommand{\fg}{''}
\providecommand{\smfandname}{\&}
\providecommand{\smfedsname}{\'eds.}
\providecommand{\smfedname}{\'ed.}
\providecommand{\smfmastersthesisname}{M\'emoire}
\providecommand{\smfphdthesisname}{Th\`ese}

\end{document}